\documentclass[12pt, reqno]{amsart}

\evensidemargin
\oddsidemargin
\makeatletter\@addtoreset{equation}{section}\makeatother

\usepackage{amsmath, amsthm, latexsym, amssymb, bbm}
\usepackage[dvips]{color}

\usepackage{calc,  subfigure, bm}
\usepackage[dvips]{graphicx}

\newcommand{\e}{\varepsilon}

\renewcommand{\phi}{\varphi}

\newcommand{\ssup}[1] {{\scriptscriptstyle{({#1}})}}

\newcommand{\gs}{\gamma_{\rm s}}
\newcommand{\gb}{\gamma_{\rm b}}

\newcommand{\C}{\mathbb C}
\newcommand{\R}{\mathbb R}
\newcommand{\Z}{\mathbb Z}

\newcommand{\N}{\mathbb N}

\newcommand{\E}{\mathbb E}
\renewcommand{\P}{\mathbb P}

\newtheorem{theorem}{Theorem}[section]

\newtheorem{lemma}[theorem]{Lemma}

\newtheorem{prop}[theorem]{Proposition}

\def\1{{\mathchoice {1\mskip-4mu\mathrm l}      % Blackboard bold 1
{1\mskip-4mu\mathrm l}
{1\mskip-4.5mu\mathrm l} {1\mskip-5mu\mathrm l}}}

\renewcommand{\subsection}{\secdef \subsct\sbsect}
\newcommand{\subsct}[2][default]{\refstepcounter{subsection}
\vspace{0.15cm}
{\flushleft\bf \arabic{section}.\arabic{subsection}~\bf #1  }
\nopagebreak\nopagebreak}
\newcommand{\sbsect}[1]{\vspace{0.1cm}\noindent
{\bf #1}\vspace{0.1cm}}

\newcounter{remnr}
\newenvironment{remark}{\refstepcounter{remnr}
{\sf Remark~\arabic{remnr}.\ }\nopagebreak  }%
{\nopagebreak {\hfill{$\diamond$}}\\ }

\renewcommand{\phi}{\varphi}

\renewcommand{\P}{\mathbb{P}}
\renewcommand{\E}{\mathbb{E}}

\newcommand{\ZZ}{\mathcal{Z}}
\newcommand{\WW}{\mathcal{W}}
\newcommand{\VV}{\mathcal{V}}

\begin{document}

\title
%[Localisation and ageing in the PAM with Weibull potential]
{Small deviations of a Galton--Watson process with immigration}% and small martingale limit}

\author[Nadia Sidorova]{}

\maketitle

\centerline{\sc Nadia Sidorova\footnote{Department of Mathematics, University College London, Gower Street, London WC1 E6BT, UK, {\tt n.sidorova@ucl.ac.uk}.\\
The author was supported by
the Leverhulme Research Grant RPG-2012-608.
} }

\vspace{0.4cm}

% \centerline{\small(\version)}

\vspace{0.4cm}

\begin{quote}
{\small {\bf Abstract:} We consider a Galton--Watson process with immigration $( \mathcal{Z}_n)$, 
with offspring probabilities $(p_i)$ and immigration probabilities $(q_i)$. In the case when $p_0=0$, $p_1\neq 0$, $q_0=0$
(that is, when $\text{essinf} (\mathcal{Z}_n)$ grows linearly in $n$), 
we establish the asymptotics of the left tail $\P\{\mathcal{W}<\e\}$, as $\e\downarrow 0$, of the martingale limit 
$\mathcal{W}$ of the process $( \mathcal{Z}_n)$. Further, we consider the first generation $\mathcal{K}$ such that 
$\mathcal{Z}_{\mathcal{K}}>\text{essinf} (\mathcal{Z}_{\mathcal{K}})$ and study the asymptotic behaviour of $\mathcal{K}$
conditionally on $\{\mathcal{W}<\e\}$, as $\e\downarrow 0$. We find the scale at which $\mathcal{K}$ goes to infinity and 
describe the fluctuations of $\mathcal{K}$ around that scale. Finally, we compare the results with those for standard Galton--Watson 
processes. 
}
\end{quote}
\vspace{5ex}

{\small {\bf AMS Subject Classification:} Primary 60J80.
Secondary 60F10.

}
\vspace{4ex}

%\renewcommand{\thefootnote}{}
%\footnote{\textit{AMS Subject Classification:} Primary 60H25
%Secondary 82C44, 60F10.}
%\footnote{\textit{Keywords: } parabolic Anderson model, Anderson Hamiltonian, random potential, intermittency, localisation,  
%Weibull tail, Weibull distribution, Feynman-Kac formula. }

%\renewcommand{\thefootnote}{1}

\section{Introduction and main results}
\medskip

\subsection{Galton--Watson process without immigration} 
\medskip

\noindent
Let $(Z_n:n\ge 0)$ be a supercritical Galton--Watson process with
a non-degenerate offspring random variable $X$. We denote the   
offspring probabilities by $(p_k: k\ge 0)$ and assume that $p_0=0$. 
Further, we assume that the process starts with one ancestor, i.e.\ $Z_0=1$, 
and denote by $a=\E X>1$ the average offspring number.
% the offspring random variable $X$ is non-degenerate 
%and satisfies
%% a non-degenerate 
%%offspring random variable $X$ satisfying
%$\P\{X=0\}=0$ and $\E X\log X<\infty$. We assume that $Z_0=1$ and denote by $a=\E X>1$ the average offspring number.
%%We also assume that $\P\{X=k\}<1$ for all $k$ in order to exclude the deterministic case. 
\medskip

It is well known by  the Kesten--Stigum
%\footnote{More precise reference perhaps?} 
theorem 
that under the condition ${\E X\log X<\infty}$ the \emph{martingale limit} 
\begin{align*}
W=\lim_{n\to\infty}\frac{Z_n}{a^n} 
\end{align*}
exists and is strictly positive almost surely. Moreover, the random variable $W$ has a strictly positive 
continuous density, see~\cite{AN}.  
However, only in a very limited number of examples can
%\footnote{Do I want to add 
%references here? There are some in Li paper and some in Wachtel} 
the distribution of $W$ be computed explicitly, and one has to rely on asymptotic 
results to describe the behaviour of $W$. 
\medskip

The \emph{left tail asymptotics}  $\P\{W<\e\}$ as $\e\downarrow 0$ of $W$ has 
attracted a lot of mathematical attention, both 
in its own right and in the broader 
context of \emph{small value problems}, see~\cite{lifshits}. 
Since small values of the martingale limit $W$ correspond to sub-average branching of the Galton--Watson tree, 
one naturally has to distinguish between the~\emph{Schr\" oder case} when $p_1>0$
and the \emph{B\" ottcher case} when  $p_1=0$. In the Schr\" oder case, small values of $Z_n$ are achieved
much more easily and, in particular, the minimal value of $Z_n$ is
equal to one with positive probability. In sharp contrast to that, in the B\" ottcher case the minimal tree grows
exponentially.   
\medskip

The Schr\" oder case was first studied in~\cite{dubuc2}, where it was shown that 
\begin{align*}
\E e^{-sW}\asymp s^{-\tau}\qquad \text{ as }s\to\infty,
\end{align*} 
with $\tau=-\frac{\log p_1}{\log a}$. Shortly after that in~\cite{dubuc} it was proved that the density $w$ of $W$
decays at zero as $w(\e)\asymp \e^{\tau-1}$, which easily implies 
\begin{align*}
\P\{W<\e\}\asymp \e^{\tau}\qquad\text{ as }\e\downarrow 0.
\end{align*}
The asymptotics for the density $w$ was then refined in~\cite{bb93} to 
$w(\e) \sim \hat L(\e)\e^{\tau-1}$ 
with an analytic multiplicatively periodic function $\hat L$,
which led to  the corresponding improvement of the left tail asymptotics 
\begin{align}
\label{s}
\P\{W<\e\}\sim L(\e)\e^{\tau}\qquad\text{ as }\e\downarrow 0
\end{align} 
with another analytic multiplicatively periodic function $L$.
\medskip

%In the Schr\" oder case, the left tail asymptotics was studied in~\cite{dubuc} and~\cite{bb93}, and it was shown that 
%the density of $W$ decays polynomially at zero. 
In the B\"ottcher case, 
the situation is more complicated as the tail of $W$ decays exponentially. It was shown in~\cite{dubuc2} 
that the Laplace transform of $W$ at infinity has the logarithmic asymptotics 
\begin{align*}
\log \E e^{-s W}\asymp s^{\beta}\qquad\text{ as }s\to\infty,
\end{align*}
where $\beta=\frac{\log \mu}{\log a}$ with 
$\mu=\min\{n: p_n>0\}\ge 2$ being the minimal offspring number.
This suggested the 
logarithmic asymptotics
\begin{align}
\label{b}
\log\P\{W<\e\}\asymp \e^{-\frac{\beta}{1-\beta}}\qquad\text{ as }\e\downarrow 0
\end{align}
and a more precise result was then obtained in~\cite{bb93}, namely, that 
\begin{align*}
\phantom{aaaa}\log \P\{W<\e\}\sim -M(\e)\e^{-\frac{\beta}{1-\beta}}\qquad\text{ as }\e\downarrow 0,
\end{align*}
where $M$ is an analytic positive multiplicatively periodic function. 
%A particular example was studied numerically in , and 
A numerical example with tiny but non-trivial oscillations in $M$ was provided in~\cite{bp}, 
and an example with a constant $M$ was given in~\cite{hambly}. Finally, the full left tail asymptotics
was  computed in~\cite{FW09} to be 
\begin{align*}
\P\{W<\e\}\sim \hat M(\e)\e^{\frac{\beta}{2(1-\beta)}}\exp\big\{-M(\e)\e^{-\frac{\beta}{1-\beta}}\big\}
\qquad\text{ as }\e\downarrow 0,
\end{align*}
where both $M$ and $\hat M$ are analytic, positive, and multiplicatively periodic. 
\medskip

The precise form of the above asymptotics as well as the approach developed in~\cite{FW09} made it 
possible to understand the influence of small values of $W$ on the Galton--Watson tree. 
%Conditioned on $W<\e$, the tree will 
%In~\cite{bgms} the first branching time 
Let  
\begin{align*}
K=\min\{n: Z_n>\mu^n\}
\end{align*}
%was considered. 
be the first generation, where a vertex has more than the minimal number of offspring. 
We will call this event the first non-trivial branching of the tree. 
It was shown in~\cite{bgms} that, conditionally on $W<\e$,  
the first branching time $K$ will grow
in the Schr\" oder case as
%\begin{align*}
%\gamma(\e)=\left\{\begin{array}{ll}\displaystyle\scriptstyle
%\frac{\log (1/\e)}{\log a} & \\
%\displaystyle
%\frac{\log (1/\e)}{\log (a/\mu)} -\frac{\log\log(1/\e)}{\log a}+H(\e)& 
%\end{array}\right.
%\end{align*}
\begin{align}
\label{gammas}
\gs(\e)=\frac{\log (1/\e)}{\log a} 
\end{align}
and  in the B\" ottcher case as
\begin{align*}
%\label{gammab}
\gb(\e)=\frac{\log (1/\e)}{\log (a/\mu)} -\frac{\log\log(1/\e)}{\log a}+H(\e)
\end{align*}
for some continuous multiplicatively periodic function $H$. However, more striking are the 
fluctuations of $K$. It was proved in~\cite{bgms} that in the B\" ottcher case there 
are \emph{no fluctuations} at all, with $K$ being equal to either $\lfloor \gb(\e)\rfloor$ or $\lfloor \gb(\e)\rfloor+1$
with probability tending to $1$, conditionally on $W<\e$. This is no longer true for the Schr\" oder case. 
Our first result below shows that the random variable $K-\gs(\e)$
conditioned on $W<\e$ has \emph{exponentially} decaying left and right tails. 

\begin{theorem} 
\label{main0}
In the Schr\" oder case, as $x\to\infty$,  
\begin{align*}
\liminf_{\e\downarrow 0}\P\big\{K- \gs(\e)>x\, \big|\, W<\e\big\}\asymp 
\limsup_{\e\downarrow 0}\P\big\{K- \gs(\e) >x\, \big|\, W<\e\big\}\asymp p_1^x
\end{align*}
%as $x\to\infty$ 
and 
\begin{align*}
\liminf_{\e\downarrow 0}\P\big\{K-\gs(\e) <-x\, |\, W<\e\big\}\asymp 
\limsup_{\e\downarrow 0}\P\big\{K-\gs(\e) <-x\, |\, W<\e\big\}\asymp p_1^{(\lambda-1)x},
\end{align*}
%as $x\to-\infty$, 
where $\lambda=\min\{k>1: p_k\neq 0\}$.  
\end{theorem}

\begin{remark}
It was shown in~\cite{bgms} that both tails of $K-\gs(\e)$ conditioned on $W<\e$ are not heavier than 
exponential but it was not known whether this estimate is sharp for either of the tails, and if so what the correspondent 
exponents and rates are.  
To find the actual tail behaviour, we had to control the error term of the 
asymptotics~\eqref{s}.
%, which is done in Lemma~\ref{secsec}. 
\end{remark}

\begin{remark}
Under the conditioning $W<\e$ there are two competing effects influencing the behaviour of $K$: 
branching too early would lead to higher values of $W$ but on the other hand it would be probabilistically expensive 
to suppress branching for too long. Having first 
branching in generation roughly equal to $\gs(\e)$ corresponds to the optimal 
trade-off between these two effects. The right tail of $K-\gs(\e)$ corresponds to a late branching, and its decay is 
given by the probability of having just one offspring in many generations, which is exponential with exponent $p_1$. 
The left tail corresponds to an early branching, which manifests itself in the appearance of extra $(\lambda-1)$ offspring 
too early. The left tail is therefore controlled by the probability of keeping the sum of $\lambda$ (rather than one) i.i.d.\
copies of $W$ small, which explains the exponent $p^{\lambda-1}$ governing the left tail.  
\end{remark}

%The different behaviour of the fluctuations of $K$ in the Schr\" oder and B\" ottcher cases is displayed 
%in\footnote{Include picture} Figure~\ref{fig1}. 
\bigskip

%%%%%%%%%%%%%%%%%%%%%%%%%%%%%%%%%%%%%%
%%%%%%%%%%%%%%%%%%%%%%%%%%%%%%%%%%%%%%
%%%%%%%%%%%%%%%%%%%%%%%%%%%%%%%%%%%%%%

\subsection{Galton--Watson process with immigration}
\medskip

The remarkable difference in fluctuations of $K$ in the Schr\" oder and B\" ottcher cases 
is due to the fact that in the former setting the minimal tree does not grow at all, having just one offspring in every generation, 
whereas in the latter one the minimal tree grows exponentially. 
A natural question to ask is what happens if the process behaves similarly to the Galton--Watson process but
its minimum grows linearly. Galton--Watson process with immigration is a natural example of such a process.  
\medskip

%In this paper we consider a Galton--Watson process with immigration $(\ZZ_n:n\ge 0)$.
Following the definition 
%in~\cite[p.~263]{AN}, 
in~\cite{AN}, we fix a non-degenerate offspring random variable $X$ with distribution $(p_k, k\ge 0)$ as before, 
%$$\Prob\{X=k\}=p_k, \qquad k\ge 0,$$ 
and an immigration random variable $Y$ with distribution $(q_k:k\ge 0)$. 
%$$\Prob\{Y=k\}=q_k, \qquad k\ge 0.$$
We define the Galton--Watson process with immigration $(\ZZ_n:n\ge 0)$ recursively by setting $\ZZ_0=Y_0$ and
\begin{align*}
\ZZ_{n+1}=X_1^{\ssup n}+\cdots+X_{\ZZ_n}^{\ssup n}+Y_{n+1}, \qquad n\ge 0, 
\end{align*}
where all $X_i^{\ssup n}$
%, $n\ge 0$, $i\ge 1$ 
are independent and identically distributed with  the same distribution as $X$, 
all $Y_j$
%, $j\ge 0$, 
are independent and identically distributed with the same distribution as $Y$, 
and all $X_i^{\ssup n}$ and $Y_j$ are independent. In other words, the Galton--Watson 
process with immigration $(\ZZ_n)$
differs from the ordinary Galton--Watson process with offspring probabilities $(p_k)$ by  
the property that, in generation $n$, there is an immigration of a random number of individuals into the population
governed by immigration probabilities $(q_k)$ and independent of the rest of the process. 
\medskip

As before, we assume that $p_0=0$. We also assume that $p_1>0$ as otherwise the linear effect of immigration will be negligible 
with respect to the exponential growth of the population.  For the immigration probabilities, we assume that $q_0=0$ in order to avoid the extinction and sub-linear growth of the minimal tree. 
\medskip

We assume that 
\begin{align*}
%\label{cond}
\E X\log X<\infty \qquad\text{ and }\qquad \E\log Y<\infty
\end{align*}
and denote $a=\E X$, which is finite by the first condition above and greater than one since $p_0=0$. 
It is a classical result, see
%\footnote{Check the paper and the conditions.}
~\cite{Se} for example, 
that under the conditions above the limit 
\begin{align*}
\WW=\lim_{n\to\infty}\frac{\ZZ_n}{a^n}
\end{align*}
exists and is positive almost surely.
\medskip

The following logarithmic left tail asymptotic was recently computed
in~\cite{LI}. As $\e\downarrow 0$,  
\begin{align}
\label{li_ass}
\log \P\{\WW<\e\}\sim -\sigma\log^2 (1/\e), 
\end{align}
where 
\begin{align*}
\sigma=\frac{\nu\log (1/p_1)}{2\log^2 a}
\end{align*}
and
\begin{align*}
\nu=\min\{i: q_i>0\}\ge 1
\end{align*}
is the minimal number the immigration random variable can take with positive probability. As it was natural to expect, the 
left tail of $\WW$ is thinner than that of the martingale limit $W$ in the Schr\" oder case~\eqref{s} but thicker than that of $W$ 
in the B\"ottcher case~\eqref{b}. The above result was then generalised to multitype processes in~\cite{CHU}. 

\bigskip

%%%%%%%%%%%%%%%%%%%%%%%%%%%%%%%%%%
%%%%%%%%%%%%%%%%%%%%%%%%%%%%%%%%%%
%%%%%%%%%%%%%%%%%%%%%%%%%%%%%%%%%%

\subsection{Main results}
\medskip

The aim of this paper is to find the full (non-logarithmic) left tail asymptotics of $\WW$ at zero and to describe the time of the first non-trivial branching  
\begin{align*}
\mathcal{K}=\min\{n: \ZZ_n>\nu(n+1)\}
\end{align*}
conditioned on $\WW<\e$ in the limit as $\e\downarrow 0$.  In particular, we want to compare the fluctuations of $\mathcal{K}$
around its typical growth with those for Galton--Watson processes without immigration in the Schr\" oder and B\" ottcher cases.
\medskip 

Let $\omega$ be the function defined implicitly in 
a right neighbourhood of zero by
\begin{align}
\label{omega}
\omega(\e)-\log\omega(\e)+\log\log a=\log(1/\e).
\end{align}
In the sequel we will drop $\e$ in most notation and, in particular, in $\omega\equiv \omega(\e)$, if there is no risk of confusion. We will also assume that $\e$ is sufficiently small. 

%%In the sequel we will drop $\e$ in most notation and, in particular, in $\rho(\e)$.  

\begin{theorem} 
\label{main1}
As $\e\downarrow 0$, 
\begin{align}
\label{asy0}
\P\big\{\WW<\e\big\}\sim 
\exp\Big\{-\sigma\omega^2+\omega M_1(\omega)-\frac{1}{2}\log \omega+M_2(\omega)\Big\}, 
\end{align}
where $M_1$ and $M_2$ are bounded functions periodic with respect to $\omega$ with period $\log a$.  
\end{theorem}
\smallskip

\begin{remark} 
%It is easy to see that 
%\begin{align*}
%\rho=\frac{\log(1/\e)}{\log a}+\frac{\log\log (1/\e)}{\log a}-\frac{\log\log a}{\log a}+o(1)
%\end{align*}
The leading term of the asymptotics~\eqref{asy0} 
agrees with the logarithmic asymptotics~\eqref{li_ass} since 
$\omega(\e)\sim \log (1/\e)$ according to~\eqref{omega}.
\end{remark}

\begin{remark} It is of course possible to express the asymptotics of $\P\big\{\WW<\e\big\}$ in terms of $\e$ rather than $\omega$ but the formula would be too bulky and less transparent, and would require replacing $\omega$ by its asymptotic decomposition in $\e$ up to the fifth order term.
\end{remark}

Now we turn our attention to the first branching time $\mathcal{K}$. Similarly to the immigration-free case, 
the second largest number of offspring 
\begin{align*}
\lambda=\min\{i>1:p_i>0\}
\end{align*}
will play an important role in describing the behaviour of $\mathcal{K}$. 
It turns out that, conditionally on $\WW<\e$, the typical growth of $\mathcal{K}$ is given by the scaling function $\gamma\equiv \gamma(\e)$ defined 
in a right neighbourhood of zero by 
\begin{align}
\label{gamma}
\gamma(\e)=\frac{\log (1/\e)}{\log a}+\Big[\frac{1}{\log a}+\frac{1}{(\lambda-1)\log p_1}\Big]\log\log (1/\e).
\end{align}

The next theorem identifies $\gamma$ as the right scale   and describes the fluctuations 
of $\mathcal{K}$ around it conditionally on $\mathcal{W}<\e$. 

\begin{theorem} 
\label{main2}
There are constants $c_1,c_2>0$ such that 
\begin{align}
\label{c1}
\limsup_{\e\downarrow 0}\P\big\{\mathcal{K}>\lfloor \gamma(\e)\rfloor +x\, \big|\, \WW<\e\big\}&=\exp\big\{-c_1p_1^{-(\lambda-1)x}\big\}\\
\liminf_{\e\downarrow 0}\P\big\{\mathcal{K}>\lfloor \gamma(\e)\rfloor +x\, \big|\, \WW<\e\big\}&=\exp\big\{-c_2p_1^{-(\lambda-1)x}\big\}
\label{c2}
\end{align}
for all $x\in \Z$. 
\end{theorem}

\begin{remark} Comparing $\gamma$  with $\gs$ in the Schr\"oder case 
it is easy to see that the immigration can both force the process to start branching earlier
and later. The former situation occurs if $ap_1^{\lambda-1}>1$
and the latter if $ap_1^{\lambda-1}<1$.  This is not intuitively clear at the first glance as we would expect 
an earlier branching for small values of $p_1$. However, the catch is that an early branching would require suppressing a larger number 
of the subtrees to ensure $\WW<\e$, which is hard if $p_1$ is small. 
\end{remark}

\begin{remark} Theorem~\ref{main2} 
describes the fluctuations for every finite $x\in \Z$ rather than just identifying the tail behaviour as $x\to\pm\infty$. 
It immediately implies that the fluctuations are of finite order 
with a \emph{double-exponential} right tail. 
This puts the Galton--Watson process with immigration strictly between the B\"ottcher case for the standard Galton--Watson process,
where there are no fluctuations at all, and the Schr\" oder case, where the right tail decays exponentially by Theorem~\ref{main0}. 
At the same time, the left tail still has an \emph{exponential} decay since 
\begin{align*}
1-\exp\big\{-c_ip_1^{-(\lambda-1)x}\big\}\asymp p_1^{-(\lambda-1)x}
\end{align*}
as $x\to -\infty$ for $i=1,2$. Moreover, according to Theorem~\ref{main0}, the exponent of decay of the left tail is exactly the same as in the Schr\" oder case, which shows that, unlike the right tail, the left tail is not affected by the immigration.  
\end{remark}

The comparison between the fluctuations in all three cases is given on the picture below.
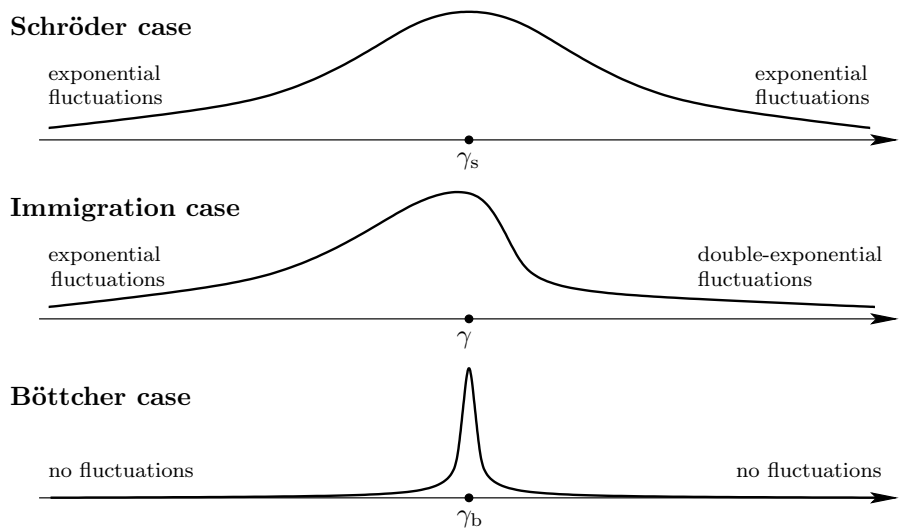
\begin{figure}[h]
\input{fluc.pstex_t}
\end{figure}

\subsection{Ideas of the proofs}
\medskip

It is well-known that the left tail of a positive random variable at zero is closely related to its Laplace transform  
at infinity. Following the approach suggested in~\cite{FW09}, we use the precise inversion formula to obtain 
$\P\{\WW<\e\}$ from the Laplace transform $\phi_*(z)=\E e^{z\WW}$. 
This is a technically challenging  step but, unlike the standard large deviations techniques, it is capable of 
providing the full asymptotics of $\P\{\WW<\e\}$ rather than the logarithmic one. 
Then we use the immigration mechanism to relate $\phi_*(z)$ with the Laplace transform 
$\phi(z)=\E e^{zW}$ of the immigration-free Galton--Watson process. 
Further, we rely on the well-known Poincar\' e functional equation 
\begin{align}
\label{pfe}
\phi(za^n)=f_n(\phi(z)),
\end{align}
where $f_n$ is the $n$-th iterate of the generating function of the offspring random variable $X$, in order
to understand the behaviour of $\phi(z)$  for large $z$ through the asymptotics of $f_n$ as $n$ tends to infinity. 
\medskip

The paper is organised as follows. In Section~\ref{s_not} we introduce notation and prove a couple of technical results for standard Galton--Watson processes. In Section~\ref{s_flu} we establish the error term of the asymptotics~\eqref{s} and use it to prove Theorem~\ref{main0}. In Section~\ref{s_imm} we introduce notation relevant to immigration and get some preliminary asymptotic results for Galton--Watson processes with immigration. In Section~\ref{s_fin} we prove two technical lemmas describing the behaviour of the Schr\"oder function. In Section~\ref{s_main} we reduce the problem of describing the behaviour of 
$\mathcal{W}$ and $\mathcal{K}$ to that of understanding the left tails of a certain sequence of random variables $(\mathcal{V}_n)$, and establish some asymptotic properties of their Laplace transforms. In Section~\ref{s_lef} we study the left tails of $(\mathcal{V}_n)$
and, finally, in Section~\ref{s_pro} we combine everything and prove Theorems~\ref{main1} and~\ref{main2}.          

%Throughout the paper, we will introduce\footnote{Do it here or when we introduce $\rho$?} 
%various scaling functions of $\e$ such as $\gamma$, $\rho$ etc. 
%We will usually drop the argument $\e$ if there is no risk of confusion. 

%\begin{align*}
%\liminf_{\e\downarrow 0}\P\big\{K\le \lfloor \gamma(\e)\rfloor -x\, \big|\, \WW<\e\big\}
%&=1-\limsup_{\e\downarrow 0}\P\big\{K>\lfloor \gamma(\e)\rfloor -x\, \big|\, \WW<\e\big\}\\
%&=1-\exp\big\{-c_1p_1^{-(\lambda-1)x}\big\}\sim c_1p_1^{-(\lambda-1)x}
%\end{align*}  
%and, similarly, 
%\begin{align*}
%\limsup_{\e\downarrow 0}\P\big\{K\le \lfloor \gamma(\e)\rfloor -x\, \big|\, \WW<\e\big\}
%\sim c_2p_1^{-(\lambda-1)x}
%\end{align*}  
%as $x\to \infty$. 

\bigskip

%%%%%%%%%%%%%%%%%%%%%%%%%%%%%
%%%%%%%%%%%%%%%%%%%%%%%%%%%%%
%%%%%%%%%%%%%%%%%%%%%%%%%%%%%

\section{Notation and preliminaries in immigration-free case}
\label{s_not}
\medskip

For any $r>0$, let $\mathcal{D}_r=\{z\in \C:|z|\le r\}$ be the closed disc of radius $r$. 
%the complex ball of radius $r$ centred at zero. 
Denote by
\begin{align*}
f(z)=\sum_{n=1}^{\infty}p_nz^n,\qquad z\in \mathcal{D}_1,
\end{align*}
the generating function of the offspring random variable $X$. 
Denote by
\begin{align*}
\phi(z)=\E e^{-zW}, \qquad z\in\C, \mathcal{R}e\, z\ge 0, 
\end{align*}
the Laplace transform of the martingale limit $W$ of the corresponding Galton--Watson tree. 
\medskip

Let $f_0(z)=z$ and, for each $n\ge 1$, 
denote $f_n(z)=f(f_{n-1}(z))$. 
%For $\delta\in (0,1)$ and 
%$\theta\in (0,\pi)$ denote 
%\begin{align*}
%\mathcal{D}(\delta,\theta)=\big\{z\in\C: |z|\le 1-\delta, |\text{arg}\, z|\le \theta\big\}. 
%\end{align*}
%
Since the behaviour of the iterates of  $f(z)$ for large $n$ is mainly determined by its leading term $p_1 z$, 
it is convenient to use the decomposition 
%Since $p_0=0$ we have for all $n$ and $z$ %$z\in \mathcal{D}_1$
\begin{align}
\label{fff}
f_n(z)=f(f_{n-1}(z))
&=p_1 f_{n-1}(z)\Big(1+p_1^{-1}\sum_{l>1}p_lf_{n-1}^{l-1}(z)\Big)
=p_1^n z\prod_{j=0}^{n-1}A_j(z),
\end{align}
where the functions $A_j$ are defined on $\mathcal{D}_1$ by 
\begin{align}
\label{defa}
A_j(z)=1+p_1^{-1}\sum_{l>1}p_lf_j^{l-1}(z).%, \qquad k\ge 0. %, z\in \C, |z|\le 1.
\end{align}
For each $|z|<1$, denote
\begin{align}
\label{infps}
S(z)=z\prod_{j=0}^{\infty}A_j(z).
\end{align}
It is well-known (see~\cite[Lemma 3.7.2 and Corollary 3.7.3]{AH}) that this infinite product converges uniformly on each $\mathcal{D}_r$, $r\in (0,1)$, 
and the function $S$ is called the Schr\" oder function. It is easy to see that 
\begin{align}
\label{defS}
S(z)=\lim_{n\to\infty} \frac{f_n(z)}{p_1^n}. 
\end{align}
In particular, on each $\mathcal{D}_r$, $r\in (0,1)$, the Schr\" oder function $S$ is bounded and 
$A_j(z)\to 1$ uniformly as $j\to \infty$.
Denote 
\begin{align*}
\mathcal{B}=\{z\in \C: |z|<1, S(z)\neq 0\}.
\end{align*}
For each $n\ge 0$ and $|z|<1$, denote 
\begin{align}
\label{defR}
R_n(z)=S(z)-p_1^{-n}f_n(z). 
\end{align}

The asymptotic behaviour of $R_n$ will be crucial for our analysis. 
In the remaining part of this section we prove some elementary properties of the functions $S$, $R_n$, and $A_n$.

\begin{lemma} 
\label{nonzero}

(a) $|S(z)|\le S(|z|)$ for any $|z|<1$. 

(b) The functions $s\mapsto S(s)$ and $s\mapsto S(s)/s$ are increasing on $[0,1)$. 

%(a) $S'(s)\ge 1$ for all $s\in [0,1)$. In particular, $S$ is increasing on $[0,1)$. 
%and\footnote{Do we need this?} $S(s)\ge s$ for all $s\in [0,1)$.
%
%(a)$(\log S(z))'\ge 1/z$ for all $z\in ()$

%(c) For any $r\in (0,1)$ there exists $\theta\in (0,\pi/2)$ such that
%$S(z)\neq 0$ on $\mathcal{D}_{r,\theta}\backslash\{0\}$, where 
%\begin{align*}
%\mathcal{D}_{r,\theta}=\{z\in\mathcal{D}_r: |\,\text{\rm arg}\, z|\le \theta\}. 
%\end{align*}
%In particular, 
%$A_k(z)\neq 0$ on $\mathcal{D}_{r,\theta}$ for all $k$. 
\end{lemma}

\begin{proof} 
(a) This follows from the same property for each $A_j$ which, in turn, follows from the same property for each $f_j$.
\medskip

(b) Obviously, it suffices to prove the second statement only. Observe that $S(s)/s$ is a product of $A_j(s)$, 
and each $A_j$ is increasing since each $f_j$ is increasing.  
%(a) Since $S$ is a uniform limit of holomorphic functions it suffices to differentiate~\eqref{defS} and observe 
%that the first term of $f'_n(s)$ is $p_1^n$ and the remaining terms are non-negative. 
%$S$ is increasing on the real interval $[0,1)$ as a product and composition of increasing functions. 
%Since $A_k(s)\ge 1$ for all $k$ on $[0,1)$ we get $S(s)\ge s$ from~\eqref{infps}.
%(c) Let $r\in (0,1)$. Observe that $S$ is holomorphic on $\mathcal{D}_r$ as a uniform limit of holomorphic functions. 
%Hence it can only have isolated zeroes, and it suffices to observe that $S(s)\neq 0$ for $s\in (0,1)$ and choose $\theta$
%sufficiently small. 
\end{proof}

\begin{lemma} 
\label{second}
(a) Let $r\in (0,1)$. Then there are constants $c_1, c_2>0$ such that  
\begin{align*}
A_n(z)-1\sim c_1 S(z)p_1^{n(\lambda-1)}
\end{align*}
and
\begin{align*}
\phantom{aaaa} R_n(z)\sim c_2 S^2(z)p_1^{n(\lambda-1)}
\end{align*}
as $n\to\infty$ uniformly on $\mathcal{D}_r\cap \mathcal{B}$. 
\medskip

(b) $R_n(z)=0$ if $z\in \mathcal{D}_r\backslash\mathcal{B}$ for all sufficiently large $n$. 
\end{lemma}

\begin{proof} 
(a)  Every convergence and equivalence mentioned in the proof below is meant to be uniform on $\mathcal{D}_r\cap \mathcal{B}$.
\medskip

Using~\eqref{fff} and~\eqref{infps} we obtain 
\begin{align*}
R_n(z)=p_1^{-n}f_n(z)\Big(\prod_{k=n}^{\infty}A_k(z)-1\Big).
\end{align*}
Since $A_j(z)\to 1$ as $j\to\infty$ and using~\eqref{defS} we have 
\begin{align}
\label{loc7}
R_n(z)\sim S(z)\Big[\exp\Big\{\sum_{j=n}^{\infty}\log A_j(z)\Big\}-1\Big].
\end{align} 
Observe that $A_j(z)\neq 1$ eventually on $\mathcal{B}$. Indeed, it follows from~\eqref{defS} that  $f_n(z)\to 0$ 
as $n\to\infty$. 
Hence the first term of the sum in~\eqref{defa} dominates over the remaining terms
\begin{align}
\label{negl}
\Big|\sum_{l>\lambda}p_l f_j^{l-1}(z)\Big|\le |f_j^{\lambda}(z)|=o(|f_j^{\lambda-1}(z)|).
\end{align} 
It remains to notice that $f_n(z)\neq 0$ eventually by~\eqref{defS} and use~\eqref{defa} together with~\eqref{negl}.  
\medskip

Now the first statement follows from 
\begin{align*}
\log A_j(z)\sim A_j(z)-1\sim \frac{p_{\lambda}}{p_1}f_j^{\lambda-1}(z)\sim p_{\lambda}p_1^{j(\lambda-1)-1}S(z)
\end{align*}
with $c_1=p_{\lambda}/p_1$, where the middle equivalence is implied by~\eqref{negl} and 
the last one by~\eqref{defS}. Now we have 
\begin{align*}
\sum_{j=n}^{\infty}\log A_j(z)\sim \sum_{j=n}^{\infty}p_{\lambda}p_1^{j(\lambda-1)-1}S(z)\sim c_2 S(z)p_1^{n(\lambda-1)}
\end{align*}
with some $c_2>0$. 
Substituting this into~\eqref{loc7} we obtains the required asymptotics. 
\medskip

(b) It is easy to see from~\eqref{infps} that $S(z)=0$ implies $z=0$ or $A_j(z)=0$ for some $j$. Then $f_n(z)=0$
eventually by~\eqref{fff} and so $R_n(z)=0$ as well. 

\end{proof}

%\begin{lemma} 
%\label{tech}
%Let $\{g_n\}_{n\in\N}$  be a sequence of complex-valued functions defined on a set $D$ 
%and let $\{c_n\}_{n\in \N}$ be a sequence of positive real numbers.  Suppose $g_n(z)\sim c_n$ as $n\to\infty$ uniformly 
%on $D$ and assume that $\sum_{n=1}^{\infty}c_n<\infty$. Then 
%\begin{align*}
%\sum_{i=n}^{\infty}g_i(z)\sim \sum_{i=n}^{\infty}c_i
%\end{align*}
%as $n\to\infty$ uniformly on $D$. 
%\end{lemma}
%
%\begin{proof} The  uniform 
%%and absolute 
%convergence of $\sum_{n=1}^{\infty}g_n(z)$ is standard. 
%For $\e>0$ let $N$ be such that $|g_i(z)/c_i-1|<\e$ for all $i\ge N$ and $x\in D$. Let $n\ge N$. Then for all $n\ge N$ and $z\in D$
%\begin{align*}
%\Big|\frac{\sum_{i=n}^{\infty}g_i(z)}{\sum_{i=n}^{\infty}c_i}-1\Big|\le \frac{\sum_{i=n}^{\infty}|g_i(z)-c_i|}{\sum_{i=n}^{\infty}c_i}<\e
%\end{align*}
%since all $c_i$ are real and positive. 
%\end{proof}

\bigskip

%%%%%%%%%%%%%%%%%%%%%%%%%%%%%%%%%%%
%%%%%%%%%%%%%%%%%%%%%%%%%%%%%%%%%%%
%%%%%%%%%%%%%%%%%%%%%%%%%%%%%%%%%%%
\section{Fluctuations in the Schr\" oder case}
\label{s_flu}
\medskip

In this section we prove Theorem~\ref{main0}. The proof for the right tail of $K-\gs$ will be rather straightforward.
The left tail, however, is controlled by the second term of the asymptotics~\eqref{s}, which we estimate in the lemma below.

\begin{lemma} 
\label{secsec}
In the Schr\" oder case, 
\begin{align*}
L(\e)\e^{\tau}-\P\{W<\e\}\asymp\e^{\tau \lambda},
\end{align*}
as 
$\e\downarrow 0$, where $\tau=-\frac{\log p_1}{\log a}$ and $L$ is an analytic multiplicatively periodic function on $(0,\infty)$ with period~$a$. 
\end{lemma}

\begin{proof}

Since $W$ is almost surely positive and has no atoms,  its left tail can be computed by the inversion formula 
\begin{align}
\label{vv}
\P\big\{W<\e\big\}=\frac{1}{2\pi}\int_{-\infty}^{\infty}\frac{1-e^{-i\tau \e }}{i\tau}\,
\phi(-i\tau)\, d\tau. 
\end{align}
%Choose $N\equiv N(\e)=\lfloor \gamma(\e)\rfloor$, where $\gamma$ is given\footnote{Bad notation.} by~\eqref{gammas}. 
Recall the definition~\eqref{gammas} of $\gs$. Changing the integration contour from the vertical axis to the vertical line 
passing through $a^{\lfloor\gs\rfloor}$ and substituting $\tau=ta^{\lfloor\gs\rfloor}$ we obtain 
\begin{align}
\label{loc8}
\P\big\{W<\e\big\}
&=\frac{1}{2\pi}\int_{-\infty}^{\infty}\frac{e^{\e a^{\lfloor\gs\rfloor}(1-i t)}-1}{1-i t}\,
\phi\big((1-i t)a^{\lfloor\gs\rfloor}\big)\, d t.
\end{align}
By the definition of $\gs$ we have $\e a^{\lfloor\gs\rfloor}=a^{-\{\gs\}}$. Further, the Poincar\' e functional equation~\eqref{pfe}
and~\eqref{defR} imply
\begin{align}
\label{loc9}
\phi\big((1-i t)a^{\lfloor\gs\rfloor}\big)=f_{\lfloor \gs\rfloor}\big(\phi(1-it)\big)
=p_1^{\lfloor \gs\rfloor}S(\phi(1-it))-p_1^{\lfloor \gs\rfloor}R_{\lfloor \gs\rfloor}(\phi(1-it)).
\end{align}
Substituting this into~\eqref{loc8} and taking into account that $p_1^{\lfloor \gs\rfloor}=\e^{\tau}p_1^{-\{\gs\}}$
by the definition of~$\tau$, we obtain 
\begin{align}
\label{loc99}
\P\big\{W<\e\big\}
&=\e^{\tau}\, \frac{p_1^{-\{\gs\}}}{2\pi }\int_{-\infty}^{\infty}\frac{e^{a^{-\{\gs\}}(1-i t)}-1}{1-i t}\,
\big[S(\phi(1-i t))-R_{\lfloor\gs\rfloor}(\phi(1-it))\big]\, d t. 
\end{align}
This representation naturally splits the the left tail probability into the leading term corresponding to $S$
and the error  term corresponding to $R_{\lfloor\gs\rfloor}$.  Namely, we define 
\begin{align}
\label{defL}
L(\e)=\frac{p_1^{-\{\gs\}}}{2\pi}\int_{-\infty}^{\infty}\frac{e^{a^{-\{\gs\}}(1-i t)}-1}{1-i t}\,S(\phi(1-it))\, d t.
\end{align}
In order to prove that the integral is finite we first observe that the ratio under the integral is bounded. Then we get by Lemma~\ref{nonzero} 
\begin{align}
\label{mono}
|S(\phi(1-it))|\le S(|\phi(1-it)|)\le \frac{S(\phi(1))}{\phi(1)}|\phi(1-it)|
\end{align}
and use~\cite[Lemma 16]{FW09} which claims that 
$|\phi(1-it)|$ is integrable over $\R$ with respect to $t$. 
Hence the function $L$ is well-defined, bounded, and multiplicatively periodic with period~$a$ since $\gs(\e)-\gs(a\e)=1\in\Z$.
In particular, once we have shown that the error term is negligible, it will imply that 
the function $L$ must be the same as in~\eqref{s}
and hence real-valued and is analytic. 
\medskip

For the error term, we use Lemma~\ref{second} with $r=\phi(1)$ to get 
\begin{align*}
\e^{\tau}\, \frac{p_1^{-\{\gs\}}}{2\pi }\int_{-\infty}^{\infty}\frac{e^{a^{-\{\gs\}}(1-i t)}-1}{1-i t}\,
R_{\lfloor\gs\rfloor}(\phi(1-it))\, d t
\asymp \e^{\tau\lambda}\int_{-\infty}^{\infty}\frac{e^{a^{-\{\gs\}}(1-i t)}-1}{1-i t}\,
S^2(\phi(1-it))\, d t, 
\end{align*}
and it remains to show that the integral on the right-hand side is bounded away from zero and infinity. 
\medskip

To do so, consider two independent random variables $W_1$ and $W_2$ with the same distribution as $W$. 
Observe that 
\begin{align}
\label{asbounds}
\e^{2\tau}\asymp  \P\big\{W_1<\e/2, W_2<\e/2\big\}
\le \P\big\{W_1+W_2<\e\big\}
\le \P\big\{W_1<\e, W_2<\e\big\}
\asymp \e^{2\tau}.
\end{align}
Similarly to~\eqref{vv}, \eqref{loc8}, \eqref{loc9}, and~\eqref{loc99} we have 
\begin{align}
\P\big\{W_1+W_2<\e\big\}
&=\frac{1}{2\pi}\int_{-\infty}^{\infty}\frac{1-e^{-i\tau \e }}{i\tau}\,
\phi^2(-i\tau)\, d\tau\notag\\
&\asymp \e^{2\tau}\int_{-\infty}^{\infty}\frac{e^{a^{-\{\gs\}}(1-i t)}-1}{1-i t}\,
\big[S(\phi(1-i t))-R_{\lfloor\gs\rfloor}(\phi(1-it))\big]^2\, d t.
\label{asy}
%\\
%&\sim \e^{2\tau}\int_{-\infty}^{\infty}\frac{e^{a^{-\{\gs\}}(1-i t)}-1}{1-i t}\,
%S^2(\phi(1-i t))\, d t
\end{align}
By Lemma~\ref{second} and boundedness of $S$ we have 
%\begin{align*}
%&\Big|\e^{2\tau}\int_{-\infty}^{\infty}\frac{e^{a^{-\{\gs\}}(1-i t)}-1}{1-i t}\,
%S(\phi(1-it))R_{\lfloor\gs\rfloor}(\phi(1-it))\, d t\Big|
%=O(1) \e^{\tau(\lambda+1)}\int_{-\infty}^{\infty}|S^3(\phi(1-it))|\, d t,\\
%&\Big|\e^{2\tau}\int_{-\infty}^{\infty}\frac{e^{a^{-\{\gs\}}(1-i t)}-1}{1-i t}\,
%R^2_{\lfloor\gs\rfloor}(\phi(1-it))\, d t\Big|
%=O(1) c \e^{2\tau\lambda}\int_{-\infty}^{\infty}|S^4(\phi(1-it))|\, d t,\\
%\end{align*}
\begin{align*}
&\Big|\int_{-\infty}^{\infty}\frac{e^{a^{-\{\gs\}}(1-i t)}-1}{1-i t}\,
S(\phi(1-it))R_{\lfloor\gs\rfloor}(\phi(1-it))\, d t\Big|
=O(\e^{\tau(\lambda-1)})\int_{-\infty}^{\infty}|S^3(\phi(1-it))|\, d t,\\
&\Big|\int_{-\infty}^{\infty}\frac{e^{a^{-\{\gs\}}(1-i t)}-1}{1-i t}\,
R^2_{\lfloor\gs\rfloor}(\phi(1-it))\, d t\Big|
=O(\e^{2\tau(\lambda-1)})\int_{-\infty}^{\infty}|S^4(\phi(1-it))|\, d t.
\end{align*}
Since $S$ is bounded, both integrals on the right-hand side are finite by the same argument 
as above in~\eqref{mono} combined with~\cite[Lemma 16]{FW09}. This implies that both terms on the left-hand side 
are $o(1)$ and hence the main term of the asymptotics~\eqref{asy} is given by 
\begin{align*}
\P\big\{W_1+W_2<\e\big\}
\asymp \e^{2\tau}\int_{-\infty}^{\infty}\frac{e^{a^{-\{\gs\}}(1-i t)}-1}{1-i t}\,
S^2(\phi(1-it))\, d t.
\end{align*}
Now~\eqref{asbounds} implies that the above integral is bounded away from zero and infinity. 
\end{proof}
\medskip

\begin{proof}[Proof of Theorem~\ref{main0}] It suffuses to prove the theorem for $x\in \Z$. 
Observe that the condition $K>\gs+x$ is equivalent to having just one offspring in the generation $\lfloor \gs\rfloor+x$. 
On this event, the condition $W<\e$ is equivalent to $\hat W<a^{\lfloor\gs\rfloor+x}\e$, where $\hat W$ is the martingale limit 
of the Galton--Watson subtree generated by that offspring. Hence $\hat W$ has the same distribution as $W$ but is 
also independent of the event $Z_{\lfloor\gs\rfloor+x}=1$. Using~\eqref{gammas} we obtain 
\begin{align*}
\P\big\{K>\gs+x, W<\e\big\}=\P\big\{Z_{\lfloor\gs\rfloor+x}=1, \hat W<a^{\lfloor\gs\rfloor+x}\e\big\}
=p_1^{\lfloor\gs\rfloor+x}\P\big\{\hat W<a^{-\{\gs\}+x}\big\}.
\end{align*}
Combining this with the left tail asymptotics~\eqref{s} and using $p_1^{\gs}=\e^{\tau}$ we get 
\begin{align*}
\P\big\{K>\gs+x\,\big|\, W<\e\big\}
&\sim \frac{p_1^{\lfloor\gs\rfloor+x}}{L(\e)\e^{\tau}}\P\big\{\hat W<a^{-\{\gs\}+x}\big\}
= \frac{p_1^{-\{\gs\}+x}}{L(a^{-\{\gs\}})}\P\big\{\hat W<a^{-\{\gs\}+x}\big\}
\end{align*}
as $\e\downarrow 0$ since $L$ is multiplicatively periodic with period $a$ and $\e=a^{-\gs}$.  
Observe that, for a fixed $x$, the expression on the right-hand side only depends on $\{\gs\}$
and so is multiplicatively periodic in $\e$. This implies that 
\begin{align*}
\liminf_{\e\downarrow 0}\P\big\{K>\gs+x\,\big|\, W<\e\big\}
&=p_1^{x}\, \frac{p_1^{-\alpha_1}}{L(a^{-\alpha_1})}\P\big\{\hat W<a^{-\alpha_1+x}\big\},\\
\limsup_{\e\downarrow 0}\P\big\{K>\gs+x\,\big|\, W<\e\big\}
&=p_1^{x}\, \frac{p_1^{-\alpha_2}}{L(a^{-\alpha_2})}\P\big\{\hat W<a^{-\alpha_2+x}\big\},
\end{align*}
for some $\alpha_1, \alpha_2\in [0,1]$. It suffices now to show that uniformly for all $\alpha\in [0,1]$
\begin{align*}
p_1^{x}\, \frac{p_1^{-\alpha}}{L(a^{-\alpha})}\P\big\{\hat W<a^{-\alpha+x}\big\}\asymp p_1^x 
\end{align*}
and 
\begin{align*}
1-p_1^{-x}\, \frac{p_1^{-\alpha}}{L(a^{-\alpha})}\P\big\{\hat W<a^{-\alpha-x}\big\}\asymp p_1^{(\lambda-1)x} 
\end{align*}
as $x\to\infty$. 
The first identity easily follows from the fact that 
$$\P\{\hat W<a^{-\alpha+x}\}\ge \P\{\hat W<a^{-1+x}\} \to 1$$ 
as $x\to\infty$. For the second one we 
observe that $a^{-\alpha-x}\downarrow 0$ as $x\to\infty$ and
so we can derive the above asymptotics from the left tail  asymptotics of the martingale limit obtained in 
Lemma~\ref{secsec}. As $L$ is multiplicatively periodic with period $a$ and $x$ is an integer we have  
\begin{align*}
1-p_1^{-x}\, \frac{p_1^{-\alpha}}{L(a^{-\alpha})}\P\big\{\hat W<a^{-\alpha-x}\big\}
&\asymp p_1^{-x}\big(L(a^{-\alpha-x})a^{-\tau(\alpha+x)}-\P\big\{\hat W<a^{-\alpha-x}\big\}\big)\\
&\asymp p_1^{-x}a^{-\tau\lambda(\alpha+x)}
\asymp p_1^{(\lambda-1)x}
\end{align*}
since $a^{-\tau}=p_1$. \end{proof}
\bigskip

%%%%%%%%%%%%%%%%%%%%%%%%%%%%%%%%%
%%%%%%%%%%%%%%%%%%%%%%%%%%%%%%%%%%
%%%%%%%%%%%%%%%%%%%%%%%%%%%%%%%%%%

\section{Immigration}
\label{s_imm}

In this section we introduce notation relevant to the immigration and prove some preliminary results which 
will be necessary to deal with it. Denote by
\begin{align*}
h(z)=\sum_{n=\nu}^{\infty}q_nz^n,\qquad z\in \mathcal{D}_1
\end{align*}
the generating function of the random variable $Y$. 
We will see in Section~\ref{s_main} that in order to find the left tail asymptotics of $\WW$ we will 
have to control the products 
\begin{align}
\label{pro}
\prod_{n=1}^N h(f_n(z))
\end{align}
for $N\in \N$. Since $f_n(z)$ will typically tend to zero as $n\to\infty$, the function $h(z)$ will essentially
behave according to its leading term $q_{\nu}z^{\nu}$. This observation suggests using the decomposition 
\begin{align}
\label{hhh_d}
h(f_n(z))=q_{\nu} f^{\nu}_n(z)\Big(1+q_{\nu}^{-1}\sum_{l>\nu}q_lf_n^{l-\nu}(z)\Big)=q_{\nu} f^{\nu}_n(z)B_n(z),
\end{align}
where the functions $B_n$ are defined by  
\begin{align}
\label{defB}
B_n(z)=1+q_{\nu}^{-1}\sum_{l>\nu}q_lf_n^{l-\nu}(z), \qquad z\in\mathcal{D}_1,
\end{align}
Combining~\eqref{fff} and~\eqref{hhh_d} we obtain 
\begin{align}
\label{sss}
h\big(f_n(z)\big)=q_{\nu}p_1^{n\nu}z^{\nu}B_n(z)\prod_{j=0}^{n-1}A_j^{\nu}(z).
\end{align}

For any $r\in (0,1)$ and $\theta\in (0,\pi/2)$, denote\begin{align*}
\mathcal{D}_{r,\theta}=\{z\in\mathcal{D}_r: z\neq 0, |\,\text{\rm arg}\, z|\le \theta\}. 
\end{align*}

\begin{lemma}
\label{ddd}
Let $r\in (0,1)$. 
As $n\to\infty$, 
\begin{align*}
\prod\limits_{j=1}^n B_j(z)\to B(z) 
\end{align*}
uniformly on $\mathcal{D}_r$, where $B$ is a bounded holomorphic function. 
\smallskip 

Further, there exists $\theta\in (0,\pi/2)$ such that on $\mathcal{D}_{r,\theta}$

(a) $S(z)\neq 0$ and, in particular, $A_k(z)\neq 0$ for all $k$; 

(b) $B(z)\neq 0$; 

(c)  
\vspace{-7ex}

\begin{align*}
\prod\limits_{j=1}^n A^{-j}_j(z)\to C(z)\neq 0
\end{align*}
uniformly as $n\to\infty$, where $C$ is a bounded holomorphic function.
\end{lemma}

\begin{proof} 
Using~\eqref{defB}
and monotonicity of $f_n$ as well as that it is bounded by $1$ on $(0,1)$ we get
\begin{align}
\label{expb}
|B_j(z)-1|\le q_{\nu}^{-1}\sum_{l>\nu}q_l |f_j^{l-\nu}(z)|\le q^{-1}_{\nu}f_j(r)\sim q_{\nu}^{-1}S(r)p_1^j
\end{align}
for all $z\in\mathcal{D}_{r}$ uniformly. Since the sum over $j$ of the expressions on the right hand side is finite, 
$\prod\limits_{j=1}^n B_j(z)$ converges uniformly on $D_r$, and hence $B$ is  
holomorphic and bounded.
\medskip

Let us now prove the second part of the lemma.
\medskip

(a) Observe that $S$ is holomorphic on $\mathcal{D}_r$ as a uniform limit of holomorphic functions. 
Hence it can only have isolated zeroes, and it suffices to observe that $S(s)\neq 0$ for $s\in (0,1)$ and choose $\theta$
sufficiently small. 
\medskip

(b)  Similarly, $B$ can only have isolated zeroes, and it is easy to see from~\eqref{defB} that 
$B(s)\ge 1$ for $s\in [0,1]$.  Hence one can choose $\theta$ small enough so that $B$ has no zeroes on 
$\mathcal{D}_{r,\theta}$. 
\medskip

(c) First we observe that by (a) the product is well defined for a sufficiently small~$\theta$. Second, 
by Lemma~\ref{second} we know that 
\begin{align*}
j\log A_j(z)\sim c S(z)j p_1^{j(\lambda-1)}
\end{align*}
uniformly as $j\to \infty$, where $S(z)\neq 0$ by (a) if $\theta$ is small enough. Again, 
the sum over $j$ of the expressions on the right hand side is finite, and we can apply the same arguments as for the first product.
\end{proof}

\begin{lemma} 
\label{hhh}

(a) Let $r\in (0,1)$. There is $c>0$ such that, for all $s\in (0,r]$, $N$ and $z\in \mathcal{D}_s$,  
%\begin{align}
%\label{sum1}
%\Big|\prod_{n=1}^{N}h\big(f_n(z)\big)\Big|
%&\le c q_{\nu}^Np_1^{\frac{\nu N(N+1)}{2}}|z|s^{\nu N}
%\prod_{j=0}^{N-1}A_j^{\nu(N-j)}(s).
%\end{align}
\begin{align}
\label{sum1}
\Big|\prod_{n=1}^{N}h\big(f_n(z)\big)\Big|
&\le c \frac{|z|}{s}\, q_{\nu}^Np_1^{-\frac{\nu N(N+1)}{2}}f_{N+1}^{\nu N}(s).
\end{align}
%uniformly 

(b) Let $r\in (0,1)$ and let $\theta$ be chosen according to Lemma~\ref{ddd}. 
Then 
\begin{align}
\prod_{n=1}^{N}h\big(f_n(z)\big)
&\sim 
F(z)q_{\nu}^Np_1^{-\frac{\nu N(N+1)}{2}}f_{N+1}^{\nu N}(z)
\label{prog2}
\end{align} 
as $N\to\infty$ uniformly on
$\mathcal{D}_{r,\theta}$, where $F$ is a bounded function on $\mathcal{D}_{r,\theta}$, which is nowhere 
equal to zero. 
%\begin{align}
%\label{defF}
%F(z)=\prod_{k=1}^{\infty}A_k^{-\nu k}(z)B_k(z). 
%\end{align}
%The infinite product above as well as 
%\begin{align}
%\label{defB}
%B(z)=\prod_{k=1}^{\infty}B_k(z)
%\end{align}
%converge uniformly on $\mathcal{D}_{r,\theta}$. 
\end{lemma}

\begin{proof}
It follows from~\eqref{sss} that, for all $N\ge 1$, we have
\begin{align}
\label{loc14}
\prod_{n=1}^{N}h\big(f_n(z)\big)
&=q_{\nu}^Np_1^{\frac{\nu N(N+1)}{2}}z^{\nu N}
\Big(\prod_{n=1}^{N}\prod_{j=0}^{n-1}A_j^{\nu}(z)\Big)
\Big(\prod_{n=1}^{N}B_n(z)\Big).
\end{align}

(a) Observe that the last term is uniformly (in $s$ and $z$) bounded by Lemma~\ref{ddd}.
Further, it follows from~\eqref{defa} that $|A_j(z)|\le  A_j(s)$ for all $j$ on $\mathcal{D}_s$. Hence 
\begin{align*}
\Big|z^{\nu N}
\prod_{n=1}^{N}\prod_{j=0}^{n-1}A_j^{\nu}(z)\Big|\le |z|s^{\nu N-1}\prod_{n=1}^{N}\prod_{j=0}^{n-1}A_j^{\nu}(s)
=|z|s^{\nu N-1}\prod_{j=0}^{N}A_j^{\nu(N-j)}(s),  
\end{align*}
where we included $j=N$ into the product since the corresponding term equals one. 
As all $A_j(s)>1$ we can drop $j$ in the exponent which together with~\eqref{fff} 
implies 
\begin{align*}
\Big|z^{\nu N}
\prod_{n=1}^{N}\prod_{j=0}^{n-1}A_j^{\nu}(z)\Big|\le\frac{|z|}{s}p_1^{-\nu N(N+1)}f_{N+1}^{\nu N}(s). 
\end{align*}
%Combining this with~\eqref{loc14} we obtain the required bound. 
%\medskip 

(b) 
%Again, the last product in~\eqref{loc14}
%converges uniformly to a bounded function $B$. 
It follows from Lemma~\ref{ddd} that all $A_j(z)\neq 0$
and so we can rearrange the middle product in~\eqref{loc14} and use~\eqref{fff} to get 
\begin{align*}
z^{\nu N}\prod_{n=1}^{N}\prod_{j=0}^{n-1}A_j^{\nu}(z)=z^{\nu N}\prod_{j=0}^{N}A_j^{\nu(N-j)}(z)
=p_1^{-\nu N(N+1)}f_{N+1}^{\nu N}(z)\prod_{j=0}^{N}A_j^{-\nu j}(z).
\end{align*}
Now the statement of the lemma follows from Lemma~\ref{ddd} with $F(z)=B(z)C^{\nu}(z)$.
%It remains to prove that the product on the right hand side of the above formula converges uniformly on $\mathcal{D}_{r,\theta}$. 
%By Lemma~\ref{second} we know that 
%\begin{align*}
%j\log A_j(z)\sim c S(z)j p_1^{j(\lambda-1)}
%\end{align*}
%uniformly as $j\to \infty$. It suffices now to observe that the sum over $j$ of the expressions on the right hand side converges, 
%and that $S$ is bounded on $\mathcal{D}_{r,\theta}$. Summarising, we obtain~\eqref{prog2} with $F(z)$ being the product of $B(z)$
%and the infinite product of $A_j^{-\nu j}(z)$ over all $j$. 
\end{proof}
\bigskip

%%%%%%%%%%%%%%%%%%%%%%%%%%%%%%%%%%%%%%%%%%
%%%%%%%%%%%%%%%%%%%%%%%%%%%%%%%%%%%%%%%%%%
%%%%%%%%%%%%%%%%%%%%%%%%%%%%%%%%%%%%%%%%%%

\section{Finer properties of the Schr\" oder function $S$}
\label{s_fin}
\medskip

\begin{lemma}
\label{logS}
The function $s\mapsto \log S(s)$ is well-defined and analytic on $(0,1)$, and 
\begin{align*}
(\log S(s))'\ge 1/s
\end{align*}
for all $s\in (0,1)$.
\end{lemma}

\begin{proof} Since $S(s)\ge s$ by~\eqref{infps} the logarithm is well defined in a complex neighbourhood of $s$, and 
the analyticity follows from $\log S$ being holomorphic there. Differentiating the uniform limit of holomorphic functions~\eqref{defS}
we obtain 
\begin{align*}
(\log S(s))'=\frac{S'(s)}{S(s)}=\lim_{n\to\infty}\frac{f'_n(s)}{f_n(s)}\ge 1/s
\end{align*} 
since this inequality is true term by term for $f_n'$ and $f_n$. 
\end{proof}

Denote 
\begin{align*}
\psi(s)=\log S(\phi(s)), \qquad s>0.
\end{align*}

\begin{lemma} 
\label{chooseu}
The function $\psi$ is analytic, $\psi''(s)>0$ for all $s$, and
\begin{align}
\label{limits}
\lim_{s\to\infty}\psi'(s)=0\qquad\text{ and }\qquad 
\lim_{s\downarrow 0}\psi'(s)=-\infty.
\end{align}
\end{lemma}

\begin{proof} The function $\psi$ is analytic as
a composition of analytic functions. Compute 
\begin{align*}
\psi'(s)&=\frac{\phi'(s)S'(\phi(s))}{S(\phi(s))}\\
\psi''(s)&=\frac{\big(\phi''(s)S'(\phi(s))+(\phi'(s))^2S''(\phi(s))\big)S(\phi(s))-\phi'(s)(S'(\phi(s)))^2}{S^2(\phi(s))}
\end{align*}
It was shown in~\cite[(75)]{FW09} that $\phi''(s)\phi(s)>(\phi'(s))^2$ for all $s>0$. 
Further, $S$ is positive according to~\eqref{infps} and $S'$ is positive by Lemma~\ref{logS}. This implies 
\begin{align}
\label{loc5}
\psi''(s)>\frac{\big( S'(\phi(s))+\phi(s)S''(\phi(s))\big)(\phi'(s))^2 S(\phi(s))-\phi(s)\phi'(s)(S'(\phi(s)))^2}{S^2(\phi(s))\phi(s)}.
\end{align}
Using~\cite[(63)]{FW09} for $f_n$ instead of $f$ we obtain  
\begin{align*}
\Big(\frac{s f_n'(s)}{f_n(s)}\Big)'>0
\end{align*}
for all $s\in (0,1)$ and all $n$, 
which extends to 
\begin{align*}
\Big(\frac{s S'(s)}{S(s)}\Big)'\ge 0
\end{align*}
for all $s\in (0,1)$ since one can differentiate uniformly convergent series of analytic functions. This implies
\begin{align*}
(S'(s)+sS''(s))S(s)\ge s (S'(s))^2
\end{align*}
for all $s\in (0,1)$. Substituting this into~\eqref{loc5} we get 
\begin{align*}
\psi''(s)>\frac{(\phi'(s)-1)\phi'(s)(S'(\phi(s)))^2}{S^2(\phi(s))}\ge 0
\end{align*}
for all $s>0$ since $\phi'(s)\le 0$. 
\smallskip 

To prove~\eqref{limits} we observe that by the Poincar\' e functional equation and~\eqref{defS} we have 
\begin{align*}
\psi(sa^n)=\log S(f_n(\phi(s)))=\log (p_1^nS(\phi(s)))=\psi(s)+n\log p_1
\end{align*}
and so
\begin{align*}
a^n \psi'(sa^n)=\psi'(s)
\end{align*}
for all $s>0$ and all $n$. Since $\psi''$ is positive $\psi'$ is decreasing and so  
\begin{align*}
&\lim_{s\to\infty}\psi'(s)=\lim_{n\to\infty}\psi'(a^n)=\psi'(1)\lim_{n\to\infty} a^{-n}=0,\\
&\lim_{s\downarrow 0}\psi'(s)=\lim_{n\to\infty}\psi'(a^{-n})=\psi'(1)\lim_{n\to\infty} a^{n}=-\infty,
\end{align*}
as required.
\end{proof}
\bigskip 

%%%%%%%%%%%%%%%%%%%%%%%%%%%%%%%%%%%%%%
%%%%%%%%%%%%%%%%%%%%%%%%%%%%%%%%%%%%%%
%%%%%%%%%%%%%%%%%%%%%%%%%%%%%%%%%%%%%%

\section{The random variables $\mathcal{V}_n$ and their Laplace transforms}
\label{s_main}
\medskip

Denote by 
\begin{align}
\label{phistar}
\phi_{*}(z)=\E e^{-z\WW}, \qquad z\in\C, \mathcal{R}e\, z\ge 0.
\end{align}
the Laplace transform of the random variable $\WW$. Observe that 
the $m$-th generation of the Galton--Watson tree with immigration  can be written as 
\begin{align}
\label{cococo}
\ZZ_n=\sum_{i=1}^{Y_0}Z_n^{\ssup{i}}+\hat{\ZZ}_{n-1},
\end{align}
where where $Z^{\ssup{i}}$ is the Galton--Watson process corresponding to the $i$-th immigrant in the generation zero, 
and $\hat{\ZZ}$ is the Galton--Watson process with immigration starting with the immigrants of generation one. 
It is easy to see that the process $\hat{\ZZ}$ and all processes $Z^{\ssup{i}}$ are independent, and $\hat{\ZZ}$
has the same distribution as $\ZZ$. Dividing by $a^n$ and passing to the limit we obtain  
\begin{align}
\label{loc10}
\WW=\sum_{i=1}^{Y_0}W_i+a^{-1}\hat{\WW},
\end{align}
where $W_{i}$ are the martingale limits of the processes $Z^{\ssup{i}}$, and $\hat{\WW}$ is the limit random variable corresponding to 
$\hat{\ZZ}$. Clearly, $\hat{\WW}$ and all $W_{i}$ are 
independent and have the same distribution as $\WW$ and $W$, respectively. This implies 
\begin{align}
\label{it}
\phi_*(z)=h(\phi(z))\phi_*(za^{-1}).
\end{align} 
%Iterating and observing that $\phi_*(a^{-n}z)\to \phi_*(0)=1$ for all $z$ we obtain the following representation of the Laplace 
%transform of $\WW$ in terms of the Laplace transorm of $W$ and the immigration mechanism $h$: 
%\begin{align*}
%\phi_*(z)=\prod_{n=0}^{\infty}h(\phi(za^{-n})).
%\end{align*}

Further, for any $k\ge -1$, denote by
\begin{align}
\label{loc11}
\VV_k=a^{-k}\sum_{i=1}^{\nu(k+1)} W^{\ssup k}_{i}+a^{-k-1}\hat{\WW}^{\ssup k},
\end{align}
where $W^{\ssup k}_i$, $1\le i\le \nu(k+1)$, are the martingale limits of independent Galton--Watson processes 
indexed by the first $\nu$ immigrants of all generations between $0$ and $k$, 
%starting at the $i$-th immigrant in generation $k$, 
and $\hat{\WW}^{\ssup k}$ is the limit random variable of the independent Galton--Watson process with immigration starting with 
the immigrants of generations strictly after $k$. Similarly to~\eqref{loc10} we observe that 
\begin{align*}
\WW=\VV_k\quad\text{ on the event }\{\mathcal{K}>k\}
\end{align*}
since the minimal number of individuals in generation $k$ is $\nu(k+1)$ given by $\nu$
immigrants in each generation and having just one offspring each.
Denote the Laplace transform of $\VV_k$ by 
\begin{align*}
\phi_k(z)=\E e^{-z\VV_k}, \qquad z\in\C, \mathcal{R}e\, z\ge 0. 
\end{align*}
It follows from~\eqref{loc11} that 
\begin{align}
\label{phik}
\phi_k(z)=\phi(za^{-k})^{\nu(k+1)}\phi_*(za^{-k-1}).
\end{align} 
%It is easy to see that\footnote{Why does the infinite product converge? By the previous lemma?}  
%\begin{align}
%\label{phi}
%\phi_k(z)
%&=\phi(za^{-k})^{\nu(k+1)}\prod_{n=k+1}^{\infty}h\big(\phi(za^{-n})\big).
%\end{align}
Observe that $\VV_{-1}=\WW$ and $\phi_{-1}=\phi_*$. We also denote $Z_{-1}=0$. 
\medskip

For each $k\ge -1$ we have  
\begin{align}
\P\big\{\mathcal{K}>k, \mathcal{W}<\e\big\}
&=\P\Big\{Z_k=\nu(k+1), \VV_k<\e\Big\}
%&=q_{\nu}^{k+1}p_1^{\frac{\nu k (k+1)}{2}}
%\P\Big\{a^{-k}\sum_{i=1}^{\nu(k+1)}W_{ki}+\sum_{n=k+1}^{\infty}a^{-n}\sum_{i=1}^{Y_n}W_{ni}<\e\Big\}\notag\\
=q_{\nu}^{k+1}p_1^{\frac{\nu k (k+1)}{2}}\P\big\{\VV_k<\e\big\}.
\label{red}
\end{align}
This means that our main aim now is to understand the left tail probabilities of $\VV_k$. 
Those corresponding to $k=-1$ will give us the left tail asymptotics of $\WW$, and those with 
$k=\gamma(\e)+x$ will control the fluctuations. Recall that the left tail of $\VV_k$ is 
closely related to the behaviour of the Laplace transform $\phi_k$ for large values of the argument. The next lemma
enables us to understand it through the asymptotic properties of the iterations $f_n$ and the branching mechanism $h$
in the same spirit as the Poincar\' e functional equation does it for a standard Galton--Watson tree.    
\medskip

For any $k$ and $N$, denote 
\begin{align*}
C_{k,N}=q_{\nu}^{N-k-1}p_1^{\frac{\nu N(N+1)-\nu k (k+1)}{2}}.
\end{align*}
Further, for any $z\in \mathcal{B}$, denote 
\begin{align*}
\Psi_{k,N}(z)=\Big(1-\frac{R_{N-k}(z)}{S(z)}\Big)^{\nu N}.
\end{align*}

%From now on, we fix 

%Now let 
%$N\equiv N(\e)=\lfloor \e\rfloor $ and $k\eqiuv k(\e)=$ 

%For all $n$ and all $z\in\mathcal{D}$ Denote\footnote{Move forward?}
%\begin{align*}
%\hat R_n(z)=1-R_n(z)/S(z)
%\end{align*}

\begin{lemma} 
\label{iter}
%Denote 
%\begin{align*}
%\hat{\mathcal{D}}_{r,\theta}=\{z\in \C: \mathcal{R}e\, z\ge 0, \phi(z)\in \mathcal{D}_{r,\theta}\}.
%\end{align*}

(a) Let $r\in (0,1)$. There is $c>0$ such that 
\begin{align}
\label{ppp}
|\phi_k(za^N)|
&\le c\frac{ |\phi(z)|}{s}\, C_{k,N}S(s)^{\nu N}.
\end{align}
for all $k<N$, $s\in (0,r]$, and all $z$ satisfying $|\phi(z)|\le s$.  
\medskip

(b) Let $r\in (0,1)$ and let $\theta$ be chosen according to Lemma~\ref{ddd}. 
Then 
\begin{align}
\label{mainass}
\phi_k(za^N)
&\sim
\Psi_{k,N}(\phi(z))\phi_*(z)F(\phi(z))  C_{k,N}
S(\phi(z))^{\nu N}
\end{align}
as $N-k\to\infty$ uniformly on $\{z: \phi(z)\in \mathcal{D}_{r,\theta}\}$.
%$\hat{\mathcal{D}}_{r,\theta}$.
\end{lemma}

\begin{proof}
Using~\eqref{phik} and iterating it according to~\eqref{it} we obtain 
\begin{align*}
\phi_k(za^N)
%&=\phi(za^{N-k})^{\nu(k+1)}\prod_{n=k+1}^{\infty}h\big(\phi(za^{N-n})\big)
=\phi_*(z)\phi(za^{N-k})^{\nu(k+1)}\prod_{n=1}^{N-k-1}h\big(\phi(za^n)\big).
\end{align*}
The Poincar\' e functional equation~\eqref{pfe} implies
\begin{align}
\label{pppp}
\phi_k(za^N)
=\phi_*(z)f_{N-k}^{\nu(k+1)}(\phi(z))\prod_{n=1}^{N-k-1}h\big(f_n(\phi(z))\big).
\end{align}

(a) Estimating the Laplace transform $\phi_*$ by one, 
the second term of~\eqref{pppp} by 
\begin{align*}
|f_{N-k}(\phi(z))|\le f_{N-k}(|\phi(z)|)\le f_{N-k}(s),
\end{align*}
and the product by Lemma~\ref{hhh} we have 
\begin{align*}
|\phi_k(za^N)|\le c\frac{ |\phi(z)|}{s}\,
q_{\nu}^{N-k-1}p_1^{-\frac{\nu (N-k-1)(N-k)}{2}}f_{N-k}^{\nu N}(s)
\end{align*} 
with some $c>0$. Taking into account
\begin{align*}
f_{N-k}(s)\le p_1^{N-k}S(s),
\end{align*}
which follows from~\eqref{fff} and~\eqref{infps} we obtain the required bound. 
\medskip

(b) Using the asymptotics for the product in~\eqref{pppp} obtained in Lemma~\ref{hhh} we get 
\begin{align*}
\phi_k(za^N)
&\sim\phi_*(z)F(\phi(z))
q_{\nu}^{N-k-1}p_1^{-\frac{\nu (N-k-1)(N-k)}{2}}f_{N-k}^{\nu N}(\phi(z))
\end{align*}
as $N-k\to\infty$ uniformly on $\{z: \phi(z)\in \mathcal{D}_{r,\theta}\}$. Taking into account 
\begin{align*}
f_{N-k}(\phi(z))=p_1^{N-k}S(\phi(z))-p_1^{N-k}R_{N-k}(\phi(z))
\end{align*}
and observing that $S(\phi(z))\neq 0$ on $\{z: \phi(z)\in \mathcal{D}_{r,\theta}\}$ by Lemma~\ref{ddd}
we arrive at~\eqref{mainass}.
\end{proof}
\bigskip

%%%%%%%%%%%%%%%%%%%%%%%%%%%%%%%%%%%%%%%%%
%%%%%%%%%%%%%%%%%%%%%%%%%%%%%%%%%%%%%%%%%
%%%%%%%%%%%%%%%%%%%%%%%%%%%%%%%%%%%%%%%%%

\section{Left tail of $\VV_k$}
\label{s_lef}
\medskip 

The aim of this section is to compute the left tail asymptotics for $\VV_k$ for two types of $k$. The first case is simply 
\begin{align}
\label{case1}
k=-1.
\end{align}
Combined with~\eqref{red}, this would give us the left tail asymptotics of $\WW$ and prove Theorem~\ref{main1}.
The second case is 
\begin{align}
\label{case2}
k\equiv k(\e,x)=\lfloor\gamma(\e)\rfloor+x
\end{align} 
for a fixed integer $x$. This is needed to prove Theorem~\ref{main2}. It turns out that both cases can be handled simultaneously 
so in this section we always assume that $k$ satisfies either~\eqref{case1} or~\eqref{case2}. 
\medskip

Let $\rho$
%: (0, 1/(e\log a))\to (1/(\log a),\infty)$ 
be the function defined implicitly in a right neighbourhood of zero by
\begin{align}
\label{rho}
\rho(\e)a^{-\rho(\e)}=\e.
\end{align}
It is easy to see that 
\begin{align}
\label{rhoomega}
\omega=\rho\log a.
\end{align}
and that the first three leading terms of the asymptotics of $\rho$ are given by  
\begin{align}
\label{fullass}
%\rho(\e)=\frac{1}{\log a}\Big[\log(1/\e)+\log\log(1/\e)-\log\log a+\frac{\log\log(1/\e)}{\log (1/\e)}-\frac{\log\log a}{\log(1/\e)}+
%o\Big(\frac{1}{\log(1/\e)}\Big)\Big]
\rho(\e)=\frac{1}{\log a}\Big[\log(1/\e)+\log\log(1/\e)
-(1+o(1))\log\log a\Big].
\end{align}
It is worth mentioning that the definition~\eqref{gamma} of 
$\gamma$ manifested itself from the condition 
\begin{align}
\label{gammarho}
\rho(\e)\, p_1^{(\lambda-1)\left(\rho(\e)-\gamma(\e)\right)}\asymp 1 \qquad\text{ as }\e\downarrow 0,
\end{align}
which will prove to be crucial later on. 
\medskip

Now fix
\begin{align*}
N\equiv N(\e)=\lfloor \rho(\e)\rfloor.
\end{align*} 
and choose $u\equiv u(\e)$ in such a way that 
\begin{align}
\label{defu}
\nu\psi'(u)=-a^{-\{\rho\}}.
\end{align}
This is possible by Lemma~\ref{chooseu} since $\psi'$ takes all negative values.  Moreover, since the right hand side of~\eqref{defu}
is bounded between $-1$ and $-1/a$ and $\psi'$ is decreasing by Lemma~\ref{chooseu}, 
there exist 
positive and independent of $\e$ constants $u_*$ and $u^*$ such that $u\in [u_*, u^*]$ for all sufficiently small $\e$.
\medskip

Similarly to~\eqref{vv}, the lower tail of $\VV_k$ can be computed by the inversion formula 
\begin{align*}
%\label{vv}
\P\big\{\VV_k<\e\big\}=\frac{1}{2\pi}\int_{-\infty}^{\infty}\frac{1-e^{-i\tau \e }}{i\tau}\,
\phi_k(-i\tau)\, d\tau. 
\end{align*}
Similarly to~\eqref{loc8}, we will move the integration contour far to the right and rescale the integration accordingly. 
However, we have to do it more carefully. Namely, we replace the vertical coordinate axis by the vertical line passing 
 through $ua^N$ and substitute $\tau=ta^N$. This gives  
\begin{align}
\label{conour}
\P\big\{\VV_k<\e\big\}=\frac{1}{2\pi}\int_{-\infty}^{\infty}\frac{e^{\e a^N(u-i t)}-1}{u-i t}\,
\phi_k\big((u-i t)a^N\big)\, d t. 
\end{align}

In order to understand the impact of the term $\phi_k$ in the integral, we will combine Lemma~\ref{iter}
with the saddle  point approximation for $S$. 
Choose $r$ in such a way that $\phi(u_*)<r<1$ and choose $\theta$ according to Lemma~\ref{ddd}.  
The composition $S\circ \phi$ maps the interval $[u_*,u^*]$ to $[S(\phi(u^*)), S(\phi(u_*))]$
since $\phi$ is decreasing on $[0,\infty)$ and $S$ is increasing on $[0,1)$ by Lemma~\ref{nonzero}. 
 Hence we can choose $\beta_1>0$ small enough so that  $\phi$ maps 
$[u_*,u^*]\times [-\beta_1,\beta_1]$ to  $\mathcal{D}_{r,\theta}$
and $S\circ \phi$ maps 
$[u_*,u^*]\times [-\beta_1,\beta_1]$ to  $\{z\in\C:\mathcal{R}e\, z>0\}$. 
Now the function $\log (S\circ \phi)$ is well-defined on $[u_*,u^*]\times [-\beta_1,\beta_1]$
and its third order derivatives are bounded. 
Expanding into the Taylor series
\begin{align}
\label{taylor}
\log S(\phi(u-it))=\log S(\phi(u))-it\psi'(u)-\frac{t^2}{2}\psi''(u)+O(t^3)
\end{align}
as $t\to 0$ uniformly in $\e$ and using the fact that $\psi''$ is positive by Lemma~\ref{chooseu},
we choose $\beta\in (0,\beta_1)$ in such a way that  
\begin{align}
\label{taylorest}
\mathcal{R}e\big[\log S(\phi(u-it))\big]\le \log S(\phi(u))-\frac{t^2}{4}\psi''(u)
\end{align}
for all $t\in [-\beta,\beta]$ and all $\e$ small enough. Finally, we choose $\alpha\equiv\alpha(\e)$ 
so that $\alpha\downarrow 0$,  
\begin{align}
\label{alpha}
\alpha e^{N\alpha^3}\to 0\qquad\text{ and }\qquad \alpha^2 N/\log N\to\infty.
\end{align}
We assume that $\e$ is small enough so that $\alpha<\beta$.
In Lemmas~\ref{part1},~\ref{part2}, and~\ref{part3} below 
we will compute the main part of the integral~\eqref{conour}
coming from integrating over $[-\alpha,\alpha]$, and show 
that the integrals over the remaining parts 
$\{|t|\in [\alpha,\beta]\}$ and $\{|t|>\beta\}$ are negligible. 
\medskip

However, before turning our attention to the integral~\eqref{conour} we compute the asymptotics of $\Psi_{k,N}$, 
which plays a crucial r\^ ole for $\phi_k((u-it)a^N)$ according to Lemma~\ref{iter}. 

\begin{lemma} 
\label{psi}
Let $r\in (0,1)$ and let $\theta$ be chosen according to Lemma~\ref{ddd}. 
\smallskip

(a) $\Psi_{-1,N}(z)\sim 1$ as $\e\downarrow 0$ uniformly on $\mathcal{D}_{r,\theta}$.
\smallskip

(b) Let $k$ be of the form~\eqref{case2}.
%$x\in \Z$ and let\footnote{When to use $\e$ and when not?} $k\equiv k(\e,x)=\lfloor \gamma(\e)\rfloor +x$. 
Then there is  $w\equiv w(\e)$ bounded away from zero and infinity such that 
\begin{align*}
\Psi_{k, N}(z)\sim \exp\big\{-w\nu p_1^{-(\lambda-1)x}S(z)\big\}
\end{align*}
as $\e\downarrow 0$ uniformly on $\mathcal{D}_{r,\theta}$.
\end{lemma}

\begin{proof} 
Comparing the definition of $\gamma$ in~\eqref{gamma} and the asymptotics~\eqref{fullass} of $\rho$
it is easy to see that $N-k\to\infty$ as $\e\downarrow 0$ in both cases. Observe that $S(z)\neq 0$ on $\mathcal{D}_{r,\theta}$
by Lemma~\ref{ddd}. Lemma~\ref{second} now implies that 
\begin{align*}
\frac{R_{N-k}(z)}{S(z)}\sim c_1\, S(z)p_1^{(\lambda-1)(N-k)}
\end{align*}
with some $c_1>0$ and hence 
\begin{align*}
\log \Psi_{k,N}\sim -c_1\nu N  S(z) p_1^{(\lambda-1)(N-k)}
\end{align*}
as $\e\downarrow 0$ uniformly on $\mathcal{D}_{r,\theta}$
\medskip

(a) If $k=-1$ then we use $N\sim \frac{\log(1/\e)}{\log a}$  to obtain  
\begin{align*}
\lim_{\e\downarrow 0}N p_1^{(\lambda-1)(N+1)}=0. 
\end{align*}

(b) For $k=\gamma-\{\gamma\}+x$ we get 
\begin{align*}
N p_1^{(\lambda-1)(N-k)}\sim p_1^{-x(\lambda-1)} p_1^{(\lambda-1)(\{\gamma\}-\{\rho\})}\rho p_1^{(\lambda-1)(\rho-\gamma)}.
%\sim p_1^{-x(\lambda-1)}p_1^{(\{\gamma\}-\{\rho\}+1)(\lambda-1)}/\log a
\end{align*}
Since $\rho p_1^{(\lambda-1)(\rho-\gamma)}\sim c_2$ according to~\eqref{gamma} and the three leading terms of $\rho$
given by~\eqref{fullass},
we obtain 
the required asymptotics with  $w\equiv w(\e)=c_1c_2p_1^{(\lambda-1)(\{\gamma\}-\{\rho\})}$. 
\end{proof}

%For each $k\ge -1$, define \footnote{some dependence on $\e$ is also in $N$. Bad notation.
%Also make sure $k$ has not been fixed yet} 

\begin{lemma} 
\label{part1}
Suppose that $k$ is of the form~\eqref{case1} or~\eqref{case2}. Then 
%either $k=-1$ or $k\equiv k(\e)=\lfloor \gamma(\e)\rfloor+x$ for some fixed $x\in \Z$. 
%Then 
\begin{align*}
&\frac{1}{2\pi}\int_{-\alpha}^{\alpha}\frac{e^{\e a^N(u-i t)}}{u-i t}\,
\phi_k\big((u-i t)a^N\big)\, d t\\
&\sim \Phi_{k,N} C_{k,N}
\exp\Big\{u\e a^N+\nu N\log S(\phi(u))-\frac 1 2 \log N\Big\}
\end{align*}
as $\e\downarrow 0$, where
%\footnote{Do we define it for $k, N$ or $\e$?} 
\begin{align*}
\Phi_{k,N}\equiv \Phi_{k,N}(\e)=
\frac{\phi_*(u)F(\phi(u))}{u\sqrt{2\pi \nu\psi''(u)}}\Psi_{k,N}(\phi(u)).
\end{align*}
\end{lemma}

\begin{proof} By our choice of 
%$\alpha$ and 
$\beta$ we have
$\phi(u-it)\in \mathcal{D}_{r,\theta}$
for all $t\in [-\beta,\beta]$. 
By Lemma~\ref{iter} we obtain
\begin{align}
\label{glo}
&\phi_k((u-it)a^N)
\sim
\Psi_{k,N}(\phi(u-it))\phi_*(u-it)F(\phi(u-it))  C_{k,N}
S(\phi(u-it))^{\nu N}
\end{align}
as $\e\downarrow 0$ uniformly for all $t\in [-\beta,\beta]$. Taking into account that  $\alpha\downarrow 0$ we have 
\begin{align*}
%\label{mainass}
&\phi_k((u-it)a^N)\sim
\Psi_{k,N}(\phi(u))\phi_*(u)F(\phi(u))  C_{k,N}
S(\phi(u-it))^{\nu N}
\end{align*}
as $\e\downarrow 0$ for all $t\in [-\alpha,\alpha]$ since $\Psi_{k,N}$ is regular enough by Lemma~\ref{psi}. 
We obtain 
\begin{align}
\label{loc2}
\int_{-\alpha}^{\alpha}&\frac{e^{\e a^N(u-i t)}}{u-i t}\,
\phi_k\big((u-i t)a^N\big)\, d t\notag\\
&\sim C_{k,N}\Psi_{k,N}(\phi(u))
\frac{\phi_*(u)F(\phi(u))}{u}
\int_{-\alpha}^{\alpha}e^{\e a^N(u-i t)}
S(\phi(u-it))^{\nu N} \, d t.
\end{align}
For the integral above, we use the Taylor expansion~\eqref{taylor} 
%and the definition~\eqref{defu} of $u$ 
to obtain
\begin{align*}
&\int_{-\alpha}^{\alpha}e^{\e a^N(u-i t)}S(\phi(u-it))^{\nu N} \, d t\\
%&\int_{-\alpha}^{\alpha}\exp\big\{\e a^N(u-i t)+\nu N\log S(\phi(u-it))\big\}\, dt\\
&\sim \exp\big\{u\e a^N+\nu N\log S(\phi(u))\big\}\int_{-\alpha}^{\alpha}
\exp\Big\{-it\big(\e a^{N}+\nu N\psi'(u)\big)-\frac{t^2 \nu N \psi''(u)}{2}\Big\}\, dt,
%\\
%&\sim \exp\big\{u\e a^N+\nu N\log S(\phi(u))\big\}\int_{-\alpha}^{\alpha}
%\exp\Big\{-it\nu\{\rho\}\psi'(u)-\frac{t^2 \nu N \psi''(u)}{2}\Big\}\, dt.
\end{align*}
where we have also used $\alpha e^{N\alpha^3}\to 0$ from~\eqref{alpha} to get rid of the negligible terms in the Taylor expansion. 
Observing that by~\eqref{rho} and~\eqref{defu}
\begin{align*}
\e a^{N}+\nu N\psi'(u)
&=\e a^{\rho-\{\rho\}}+\nu \rho \psi'(u)-\nu\{\rho\}\psi'(u)\\
&=\rho\big(a^{-\{\rho\}}+\nu\psi'(u)\big)-\nu\{\rho\}\psi'(u)=-\nu\{\rho\}\psi'(u)
\end{align*}
we obtain 
\begin{align}
&\int_{-\alpha}^{\alpha}e^{\e a^N(u-i t)}S(\phi(u-it))^{\nu N} \, d t \notag\\
&\sim \exp\big\{u\e a^N+\nu N\log S(\phi(u))\big\}\int_{-\alpha}^{\alpha}
\exp\Big\{it\nu\{\rho\}\psi'(u)-\frac{t^2 \nu N \psi''(u)}{2}\Big\}\, dt.
\label{loc1}
\end{align}
Substituting $\tau=t\sqrt{\nu N \psi''(u)}$ we get 
\begin{align*}
&\int_{-\alpha}^{\alpha}
\exp\Big\{it\nu\{\rho\}\psi'(u)-\frac{t^2 \nu N \psi''(u)}{2}\Big\}\, dt\\
&\sim \frac{1}{\sqrt{\nu N \psi''(u)}} 
%\big(\nu N \psi''(u)\big)^{-\frac{1}{2}}
\int_{-\alpha\sqrt{\nu N \psi''(u)}}^{\alpha\sqrt{\nu N \psi''(u)}}
\exp\Big\{\frac{i\tau\nu\{\rho\}\psi'(u)}{\sqrt{\nu N \psi''(u)}}-\frac{\tau^2 }{2}\Big\}\, d\tau
\sim \frac{\sqrt{2\pi}}{\sqrt{\nu N \psi''(u)}}
\end{align*}
since, as $\e\downarrow 0$, the interval of integration increases to $\R$ by~\eqref{alpha}
and the first order term tends to zero (see for example~\cite[Lemma 11]{bgms}).
Combining this with~\eqref{loc1} we obtain 
\begin{align*}
&\int_{-\alpha}^{\alpha}e^{\e a^N(u-i t)}S(\phi(u-it))^{\nu N} \, d t \\
&\sim  \frac{\sqrt{2\pi}}{\sqrt{\nu\psi''(u)}}\exp\Big\{u\e a^N+\nu N\log S(\phi(u))-\frac 1 2 \log N\Big\}.
\end{align*}
Together with~\eqref{loc2} this proves the required asymptotics. 
\end{proof}

\begin{lemma} 
\label{part2}
Suppose that $k$ is of the form~\eqref{case1} or~\eqref{case2}. Then 
%Suppose that either $k=-1$ or $k\equiv k(\e)=\lfloor \gamma(\e)\rfloor+x$ for some fixed $x\in \Z$.  Then 
\begin{align*}
\Big|\int_{|t|\in [\alpha,\beta]}&\frac{e^{\e a^N(u-i t)}}{u-i t}\,
\phi_k\big((u-i t)a^N\big)\, d t\Big|\\
&\le C_{k,N}
\exp\Big\{u\e a^N+\nu N\log S(\phi(u))-\log N\Big\},
\end{align*}
for all $\e$ small enough.
\end{lemma}

\begin{proof} 
Observe that~\eqref{glo} is in particular true for all $t$ such that $|t|\in [\alpha,\beta]$. 
Since $|u-it|\ge u_*$, $|\phi_*(u-it)|\le 1$, 
$F$ is bounded by Lemma~\ref{hhh}, 
and $\Psi_{k,N}$ is uniformly bounded by Lemma~\ref{psi}, there is a positive constant $c$ such that 
\begin{align}
\Big|\int_{|t|\in [\alpha,\beta]}&\frac{e^{\e a^N(u-i t)}}{u-i t}\,
\phi_k\big((u-i t)a^N\big)\, d t\Big|\notag\\
&\le  c\, C_{k,N}
\Big|\int_{|t|\in [\alpha,\beta]}e^{\e a^N(u-i t)}
S(\phi(u-it))^{\nu N} \, d t\Big|.
\label{loc3}
\end{align}
For the integral above, we use the Taylor bound~\eqref{taylorest}
to obtain
\begin{align}
%&\Big|\int_{|t|\in [\alpha,\beta]}\exp\big\{\e a^N(u-i t)+\nu N\log S(\phi(u-it))\big\}\, dt\Big|\notag\\
&\Big|\int_{|t|\in [\alpha,\beta]}e^{\e a^N(u-i t)}S(\phi(u-it))^{\nu N} \, d t\Big|\notag\\
&\le \exp\big\{u\e a^N+\nu N\log S(\phi(u))\big\}\int_{|t|\in [\alpha,\beta]}
\exp\Big\{-\frac{t^2 \nu N \psi''(u)}{4}\Big\}\, dt\\
&\le 2\beta \exp\Big\{u\e a^N+\nu N\log S(\phi(u))-\frac{\alpha^2 \nu N \psi''(u)}{4}\Big\}.
\label{loc4}
\end{align}
Taking into account the fact that $\alpha^2 N/\log N\to \infty$ according to~\eqref{alpha} and 
combining~\eqref{loc3} with~\eqref{loc4} we obtain the desired estimate. 
\end{proof}

\begin{lemma} 
\label{part3}
Suppose that $k$ is of the form~\eqref{case1} or~\eqref{case2}.
%Suppose that either $k=-1$ or $k\equiv k(\e)=\lfloor \gamma(\e)\rfloor+x$ for some fixed $x\in \Z$. 
There exists $\delta>0$ such that
\begin{align}
\Big|\int_{|t|\ge \beta}&\frac{e^{\e a^N(u-i t)}}{u-i t}\,
\phi_k\big((u-i t)a^N\big)\, d t\Big|\notag\\
&\le C_{k,N}\exp\Big\{u\e a^N+\nu N\log S(\phi(u))-\delta\nu N\Big\}
\label{int3}
\end{align}
and 
\begin{align}
\Big|\int_{-\infty}^{\infty}&\frac{1}{u-i t}\,
\phi_k\big((u-i t)a^N\big)\, d t\Big|\notag\\
&\le C_{k,N}
\exp\Big\{\nu N\log S(\phi(u))-\delta\nu N\Big\}\phantom{u\e a^N+}
\label{int4}
\end{align}
for all $\e$ small enough.
\end{lemma}

\begin{proof}
Similarly to the proof of~\cite[Lemma 16]{FW09}, we use the fact that, for each
$v\in [u_*,u^*]$, $t\mapsto \phi(v-it)/\phi(v)$ is the characteristic function of some absolutely continuous law
(Cram\' er transform), the continuity of the mapping $(v,t)\mapsto \phi(v-it)/\phi(v)$, and the compactness of
$[u_*,u^*]$ to conclude that there is a constant $\eta$ such that
\begin{align}
\label{mon}
|\phi(u-it)|\le (1-\eta)\phi(u)\qquad \text{ for all }|t|\ge \beta
\end{align}
and all $\e$ small enough. Using Lemma~\ref{iter} with $s\equiv s(\e)=(1-\eta)\phi(u)<\phi(u_*)<r$ and 
taking into account that 
$s\ge (1-\eta)\phi(u^*)$ and so is separated from zero,  we obtain, 
with some positive constant $c_1$, 
\begin{align*}
%\label{ppp}
|\phi_k((u-it)a^N)|
&\le c_1|\phi(u-it)|C_{k,N}
S\big((1-\eta)\phi(u)\big)^{\nu N}.
%\Psi_{k,N}\big((1-\eta)\phi(u)\big)
\end{align*}
By Lemma~\ref{logS} and the mean value theorem
\begin{align*}
\log S\big((1-\eta)\phi(u)\big)\le \log S\big(\phi(u)\big)-\eta,  
\end{align*}
which implies 
\begin{align*}
|\phi_k((u-it)a^N)|
&\le c_1|\phi(u-it)|
%q_{\nu}^{N-k-1}
C_{k,N}
S\big(\phi(u)\big)^{\nu N}e^{-\eta \nu N}.
\end{align*}
Substituting this estimate into the integral~\eqref{int3} and using $|u-it|\ge u_*$ we obtain 
\begin{align}
&\Big|\int_{|t|\ge \beta}\frac{e^{\e a^N(u-i t)}}{u-i t}\,
\phi_k\big((u-i t)a^N\big)\, d t\Big|\notag\\
&\le c_2\, C_{k,N}
\exp\Big\{u\e a^N+\nu N\log S(\phi(u))-\eta\nu N \Big\}\int_{-\infty}^{\infty}|\phi(u-it)|dt
\label{loc6}
\end{align}
with some $c_2>0$.
It was shown in~\cite[Lemma 16]{FW09} that the integral above is uniformly bounded. This implies the
required estimate with some $\delta<\eta$. 
\smallskip

The estimate~\eqref{int4} is obtained in the same way as~\eqref{int3} with the only difference that the terms
$e^{\e a^N(u-i t)}$ and $u\e a^N$ are omitted. 
\end{proof}
\bigskip

%%%%%%%%%%%%%%%%%%%%%%%%%%%%%%%%%%%%%%%%%
%%%%%%%%%%%%%%%%%%%%%%%%%%%%%%%%%%%%%%%%%
%%%%%%%%%%%%%%%%%%%%%%%%%%%%%%%%%%%%%%%%%

\section{Proofs of the main theorems}
\label{s_pro}
\medskip

In this section we establish the joint probability~\eqref{red} and then prove Theorems~\ref{main1} and~\ref{main2}. 

\begin{prop} 
\label{prop}
Suppose that $k$ is of the form~\eqref{case1} or~\eqref{case2}. 
As $\e\downarrow 0$,
\begin{align}
\label{prop1}
\P\{\mathcal{K}>k, \mathcal{W}<\e\}\sim 
q_{\nu}^{N}p_1^{\frac{\nu N(N+1)}{2}}
\Phi_{k,N}\exp\Big\{u\e a^N+\nu N\log S(\phi(u))-\frac 1 2 \log N\Big\}.
\end{align} 
\end{prop}

\begin{proof}
Combining~\eqref{red} and~\eqref{conour} we obtain 
\begin{align*}
\P\{\mathcal{K}>k, \mathcal{W}<\e\}
%&=q_{\nu}^{k+1}p_1^{\frac{\nu k (k+1)}{2}}\P\big\{\VV_{k}<\e\big\}\\
&=q_{\nu}^{k+1}p_1^{\frac{\nu k (k+1)}{2}}\frac{1}{2\pi}\int_{-\infty}^{\infty}\frac{e^{\e a^N(u-i t)}-1}{u-i t}\,
\phi_{k}\big((u-i t)a^N\big)\, d t. 
\end{align*}
Let us split the integral into the sum of the three integrals corresponding to keeping $e^{\e a^N(u-i t)}$
in the numerator and integrating over $[-\alpha,\alpha]$, $\{t:|t|\in [\alpha,\beta]\}$, and $\{t:|t|\ge \beta\}$,
respectively, and the integral corresponding to keeping $-1$ in the numerator and integrating over $\R$. 
Lemma~\ref{part1} gives the asymptotics of the first integral, while Lemmas \ref{part2} and~\ref{part3}
imply that the remaining three integrals are negligible since
$\Psi_{k,N}$ is bounded away from zero and infinity by Lemma~\ref{psi}. Substituting the asymptotics for the first integral given 
by Lemma~\ref{part1} we arrive at the required formula. 
\end{proof}

\medskip

\begin{proof}[Proof of Theorem~\ref{main1}] We use Proposition~\ref{prop} with $k=-1$ as 
$$\P\{\mathcal{W}<\e\}=\P\{\mathcal{K}>-1, \mathcal{W}<\e\}.$$ 
By Lemma~\ref{psi} we have 
\begin{align*}
\Phi_{-1,N}\sim \frac{\phi_*(u)F(\phi(u))}{u\sqrt{2\pi \nu\psi''(u)}}.
\end{align*} 
Using~\eqref{rho}
we get 
\begin{align*}
u\e a^N&=u\e a^{\rho-\{\rho\}}=u\rho a^{-\{\rho\}}\\
\nu N\log S(\phi(u))&=\nu\rho \log S(\phi(u))-\nu\{\rho\} \log S(\phi(u))\\
\log N&=\log\rho +o(1)
\end{align*}
Further, 
\begin{align*}
q_{\nu}^{N}&=\exp\big\{\rho\log q_{\nu}-\{\rho\}\log q_{\nu}\big\}\\
p_1^{\frac{\nu N(N+1)}{2}}&=\exp\Big\{\frac{\nu\log p_1}{2}\Big(\rho^2+\rho\big(1-2\{\rho\}\big)
+\{\rho\}^2-\{\rho\}\Big)\Big\}.
\end{align*}
Substituting all of the above into~\eqref{prop1}  we obtain the required asymptotics~\eqref{asy0} with
\begin{align*}
M_1(\omega)&=\frac{1}{\log a}\hat M_1\Big(\frac{\omega}{\log a}\Big),\\
M_2(\omega)&=\frac{1}{2}\log\log a+\hat M_2\Big(\frac{\omega}{\log a}\Big),
\end{align*}
where 
\begin{align*}
\hat M_1(\rho)&=\frac{\nu\log p_1}{2}\big(1-2\{\rho\}\big)+ua^{-\{\rho\}}+\nu\log S(\phi(u))+\log q_{\nu}\\
\hat M_2(\rho)&=\log\Big[\frac{\phi_*(u)F(\phi(u))}{u\sqrt{2\pi \nu\psi''(u)}}\Big]+\frac{\nu\log p_1}{2}\big(\{\rho\}^2-\{\rho\}\big)-\nu\{\rho\} \log S(\phi(u))-\{\rho\}\log q_{\nu},
\end{align*}
which are bounded periodic functions of $\rho$ with period one since $u$ is a bounded function of $\{\rho\}$ by~\eqref{defu}. 
\end{proof}
\medskip

\begin{proof}[Proof of Theorem 2] Let $x\in \Z$. Using Proposition~\ref{prop}
both with $k=-1$ and $k\equiv k(\e)=\lfloor\gamma\rfloor+x$ we obtain 
\begin{align*}
%\P\big\{\mathcal{K}>\gamma+x \, \big| \, \mathcal{W}<\e\big\}
\P\big\{\mathcal{K}>\lfloor\gamma\rfloor+x \, \big| \, \mathcal{W}<\e\big\}
&=\frac{\P\{\mathcal{K}>\lfloor\gamma\rfloor+x, \mathcal{W}<\e\}}{\P\{\mathcal{K}>-1,\mathcal{W}<\e\}}
\sim \frac{\Phi_{\lfloor\gamma\rfloor+x,N}}{\Phi_{-1,N}}\\
&=\frac{\Psi_{\lfloor\gamma\rfloor+x,N}(\phi(u))}{\Psi_{-1,N}(\phi(u))}
\sim \exp\Big\{-w \nu p_1^{-(\lambda-1)x}S(\phi(u))\Big\}
\end{align*}
by Lemma~\ref{psi}. This implies~\eqref{c1} and~\eqref{c2} with 
\begin{align*}
c_1=\nu \liminf_{\e\downarrow 0} \big[wS(\phi(u))\big]%\in (0,\infty)
\qquad\text{ and }\qquad
c_2=\nu \limsup_{\e\downarrow 0} \big[wS(\phi(u))\big]%\in (0,\infty)
\end{align*}
which are both positive and finite since $S\circ \phi$ is continuous on $[u_*,u^*]$ and $w$
is bounded away from zero and infinity by Lemma~\ref{psi}. 
\end{proof}

\end{document}